\documentclass[a4paper,11pt]{amsart}
\usepackage{amssymb}
\usepackage[OT1]{fontenc}
\usepackage[latin1]{inputenc}
\usepackage{hyperref}
\usepackage[curve]{xypic}
\usepackage{enumerate}
\usepackage[active]{srcltx}
\usepackage{ae,aecompl}

\def\Cl#1{\ensuremath{{\mathcal {#1}}}}
\def\z#1{\ensuremath{{{#1}}^{\ZZ}}}
\def\pv#1{\ensuremath{{\mathsf{#1}}}}
\def\Om#1#2{\ensuremath{\overline\Omega_{#1}{\pv{#2}}}}

\def\KerS#1{\ensuremath{\mathop{\rm Ker}{#1}}}

\def\ZZ{\ensuremath{\mathbb{Z}}}

\let\ov=\overline

\def\li#1{\ensuremath{\overleftarrow{#1}.\overrightarrow{#1}}}
\def\si#1{\ensuremath{\Sigma({\Cl #1})}}

\def\ori#1{\ensuremath{\overrightarrow{#1}}}
\def\ole#1{\ensuremath{\overleftarrow{#1}}}

\def\Mir{\ensuremath{{\operatorname{\Cl {M}}}}}

\newtheorem{Thm}{Theorem}[section]
\newtheorem{Prop}[Thm]{Proposition}
\newtheorem{Lemma}[Thm]{Lemma}
\newtheorem{Cor}[Thm]{Corollary}

{\theoremstyle{remark}
\newtheorem{Rmk}[Thm]{Remark}}

\numberwithin{equation}{section}

\begin{document}

\title[A geometric interpretation of the Sch\"utzenberger group]{A geometric interpretation of the Sch\"utzenberger group
  of a minimal subshift}

\thanks{ %
  Work partially supported respectively by CMUP (UID/MAT/00144/2013)
  and CMUC (UID/MAT/00324/2013), which are funded by FCT (Portugal)
  with national (MEC) and European structural funds through the
  programs FEDER, under the partnership agreement PT2020.}

\author{Jorge Almeida}
\address{CMUP, Departamento de Matem\'atica,
     Faculdade de Ci\^encias, Universidade do Porto, 
 Rua do Campo Alegre 687, 4169-007 Porto, Portugal.}
\email{jalmeida@fc.up.pt}

\author{Alfredo Costa}
\address{CMUC, Department of Mathematics, University of Coimbra,
  3001-501 Coimbra, Portugal.}
\email{amgc@mat.uc.pt}

\subjclass[2010]{Primary 20M05. Secondary 20E18, 37B10, 57M05, 20M50}

\keywords{free profinite semigroup, profinite group, irreducible
  subshift, minimal subshift, Rauzy graph, return word, fundamental
  group, fundamental groupoid}

\begin{abstract}
  The first author has associated in a natural way a profinite group
  to each irreducible subshift. The group in question was initially
  obtained as a maximal subgroup of a free profinite semigroup. In the
  case of minimal subshifts, the same group is shown in the present
  paper to also arise from geometric considerations involving the
  Rauzy graphs of the subshift. Indeed, the group is shown to be
  isomorphic to the inverse limit of the profinite completions of the
  fundamental groups of the Rauzy graphs of the subshift. A further
  result involving geometric arguments on Rauzy graphs is a criterion
  for freeness of the profinite group of a minimal subshift based on
  the Return Theorem of Berth\'e et.~al.
\end{abstract}

\maketitle


\section{Introduction}

The importance of (relatively) free profinite semigroups in the study
of pseudovarieties of finite semigroups is well established since the
1980's, which provides a strong motivation to understand their
structure. The algebraic-topological structure of free profinite
semigroups is far more complex than that of free semigroups. For
instance, Rhodes and Steinberg showed that the (finitely generated)
projective profinite groups are precisely the closed subgroups of
(finitely generated) free profinite
semigroups~\cite{Rhodes&Steinberg:2008}.

In the last decade, a connection introduced by the first author with
the research field of symbolic dynamics provided new insight into the
structure of free profinite semigroups, notably in what concerns their
maximal subgroups~\cite{Almeida:2005c,Almeida:2003b,Almeida:2003cshort}.
This connection is made via the languages of finite blocks of
symbolic dynamical systems, also known as
subshifts~\cite{Lind&Marcus:1996}. In symbolic dynamics,
irreducible subshifts deserve special attention: they are the
ones which have a dense forward orbit. For each irreducible subshift
$\Cl X$ over a finite alphabet $A$, one may consider the topological
closure in the $A$-generated free profinite semigroup $\Om AS$ of the
language of finite blocks of $\Cl X$. This closure is a union of \Cl
J-classes, among which there is a minimum one, $J(\Cl X)$, in the \Cl
J-ordering \cite{Almeida&ACosta:2007a}. The \Cl J-class $J(\Cl X)$
contains (isomorphic) maximal subgroups, which, as an abstract
profinite group, the authors called in~\cite{Almeida&ACosta:2013} the
\emph{Sch\"utzenberger group of $\Cl X$}, denoted $G(\Cl X)$.

The approach used in~\cite{Almeida:2005c,Almeida&ACosta:2013} consists
in obtaining information about $G(\Cl X)$ using ideas, results and
techniques borrowed from the theory of symbolic dynamical systems. The
minimal subshifts, considered in those papers, are precisely the
subshifts \Cl X for which the \Cl J-class $J(\Cl X)$ consists of
$\mathcal J$-maximal regular elements of~$\Om AS$
\cite{Almeida:2005c}.

The subshifts considered in~\cite{Almeida:2005c,Almeida&ACosta:2013}
are mostly substitutive
systems~\cite{Queffelec:1987,Fogg:2002},
that is, subshifts
defined by (weakly) primitive substitutions. Substitutive subshifts
are minimal subshifts which are described by a finite computable
amount of data, which leads to various decision problems. The authors
showed in~\cite{Almeida&ACosta:2013} how to compute from a primitive
substitution a finite profinite presentation of the Sch\"utzenberger
group of the subshift defined by the substitution, and used this to
show that it is decidable whether or not a finite group is a
(continuous) homomorphic image of the subshift's Sch\"utzenberger
group. The first examples of maximal subgroups of free profinite
semigroups that are not relatively free profinite groups were also
found as Sch\"utzenberger groups of substitutive
systems~\cite{Almeida:2005c,Almeida&ACosta:2013}.

The Sch\"utzenberger group of the full shift $A^{\mathbb Z}$ is
isomorphic to the maximal subgroups of the minimum ideal of $\Om AS$
and was first identified in~\cite{Steinberg:2009}, with techniques
that were later extended to the general sofic case
in~\cite{ACosta&Steinberg:2011} taking into account the invariance of
$G(\Cl X)$ under conjugacy of symbolic dynamical
systems~\cite{Costa:2006}. This led to the main result
of~\cite{ACosta&Steinberg:2011} that $G(\Cl X)$ is a free profinite
group with rank $\aleph_0$ when $\Cl X$ is a non-periodic irreducible
sofic subshift.\footnote{Note that the minimal sofic subshifts are the
  periodic ones.} From the viewpoint of the structure of the group
$G(\Cl X)$, the class of irreducible sofic subshifts is thus quite
different from that of substitutive (minimal) subshifts.

Substitutive systems are a small part of the realm of minimal
subshifts, in the sense that substitutive systems have zero entropy
\cite{Queffelec:1987}, while there are minimal subshifts of entropy
arbitrarily close to that of the full shift
\cite{Damanik&Solomyak:2002}. Therefore, it would be interesting to
explore other techniques giving insight on the Sch\"utzenberger group
of arbitrary minimal subshifts. That is one of the main purposes of
this paper. We do it by exploring the Rauzy graphs of subshifts, a
tool that has been extensively used in the theory of minimal
subshifts. For each subshift~$\Cl X$ and integer $n$, the Rauzy graph
$\Sigma_n(\Cl X)$ is a De Bruijn graph where the vertices (words of
length~$n$) and edges (words of length~$n+1$) not in the language of
the subshift have been removed. This graph is connected if $\Cl X$ is
irreducible. In the irreducible case, we turn our attention to the
profinite completion $\hat\Pi_n(\Cl X)$ of the fundamental group of
$\Sigma_n(\Cl X)$. The subshift $\Cl X$ can be seen in a natural way
as an inverse limit of the graphs of the form $\Sigma_{2n}(\Cl X)$.
The main result of this paper
(Corollary~\ref{c:the-geometric-interpretation-reformulation}) is that
the induced inverse limit of the profinite groups $\hat\Pi_{2n}(\Cl
X)$ is $G(\Cl X)$, provided $\Cl X$ is minimal. We leave as an open
problem whether this result extends to arbitrary irreducible
subshifts.

The study of Rauzy graphs of a minimal subshift often appears
associated with the study of sets of return words, as in the proof of
the \emph{Return Theorem}
in~\cite{Berthe&Felice&Dolce&Leroy&Perrin&Reutenauer&Rindone:2015}. We
apply the Return Theorem, together with a technical result on return
words giving a sufficient condition for freeness of the
Sch\"utzenberger group of a minimal subshift, to show that if the
minimal subshift involves $n$ letters and satisfies the so-called tree
condition~\cite{Berthe&Felice&Dolce&Leroy&Perrin&Reutenauer&Rindone:2015},
then its Sch\"utzenberger group is a free profinite group of rank $n$
(Theorem~\ref{t:tree-are-free}). This result was obtained
in~\cite{Almeida:2005c} for the important special case of Arnoux-Rauzy
subshifts, with a different approach: the result was there first
proved for substitutive Arnoux-Rauzy subshifts, and then extended to
arbitrary Arnoux-Rauzy subshifts using approximations by substitutive
subshifts.

\section{Profinite semigroups, semigroupoids, and groupoids}

\subsection{Free profinite semigroups}

We refer to~\cite{Almeida:2003cshort}
as a useful introductory text about the theory of profinite
semigroups. In~\cite{Almeida:2002a} one finds an introduction to the
subject via the more general concept of profinite algebra.
We use the notation $\Om AS$ for the free profinite
semigroup generated by the set~$A$.
Recall that $\Om AS$ is
a profinite semigroup in which $A$ embeds and which is
characterized by the property that every continuous mapping
$\varphi\colon A\to S$ into a profinite semigroup $S$ extends
in a unique way to a continuous semigroup homomorphism
$\hat\varphi\colon \Om AS\to S$.
Replacing the word ``semigroup'' by ``group'', we
get the characterization of the free profinite group with basis~$A$,
which we denote by $\Om AG$. We shall use frequently the fact that
the discrete subsemigroup of $\Om AS$ generated by $A$ is the free
semigroup $A^+$, and that its elements are the isolated elements of
$\Om AS$ (for which reason the elements of~$A^+$ are said to
be \emph{finite}, while
those in the subsemigroup $\Om AS\setminus A^+$ are \emph{infinite}).
The free group generated by $A$, denoted $FG(A)$,
also embeds naturally into $\Om AG$, but its elements are not isolated.

\subsection{Free profinite semigroupoids}

Except stated otherwise,
by a \emph{graph} we mean a directed graph with possibly multiple edges.
Formally: for us a graph is a pair of disjoint sets $V$, of \emph{vertices},
and $E$, of \emph{edges}, together with two incidence maps
$\alpha$ and $\omega$ from $E$ to $V$, the \emph{source} and the \emph{target}.
An edge $s$ with source $x$ and target $y$
will sometimes be denoted $s\colon x\to y$.
Recall that a \emph{semigroupoid} is a graph endowed with a partial
associative operation, defined on consecutive edges
(cf.~\cite{Tilson:1987,Jones:1996,Almeida&Weil:1996}):
for $s\colon x\to y$ and $t\colon y\to z$, their composite is an edge
$st$ such that $st\colon x\to z$.
Alternatively, a semigroupoid may be seen a small category
where some local identities are possibly missing.

Semigroups can be seen as being the one-vertex semigroupoids.
If the set of loops of the semigroupoid $S$ rooted at a vertex $c$ 
is nonempty, then, for the composition law,
it is a semigroup (for us an empty set is not a semigroup), the  \emph{local semigroup of $S$ at~$c$}, denoted $S(c)$.

The
theory of topological/profinite semigroups inspires a
theory of topological/profinite semigroupoids, but
as seen in~\cite{Almeida&ACosta:2007a},
there are some
differences which have to be taken into account, namely
in the case of semigroupoids with an infinite
number of vertices.
To begin with, the very definition of profinite
semigroupoid is delicate. We use the following definition: a
compact semigroupoid $S$ is \emph{profinite} if, for every pair $u,v$ of distinct
elements of $S$, there is a continuous semigroupoid homomorphism
$\varphi\colon S\to F$ into a finite semigroupoid such that
$\varphi(u)\neq\varphi(v)$. There is an unpublished example due to
G.~Bergman (mentioned in \cite{Rhodes&Steinberg:2002}) of an
infinite-vertex semigroupoid that is profinite according to this
definition, but that is not an inverse limit of finite
semigroupoids. On the other hand, it is known that
a topological graph $\Gamma$ is an inverse limit of
finite graphs if and only if for every $u,v\in\Gamma$ there is a
continuous homomorphism of graphs $\varphi\colon \Gamma\to F$ into a
finite graph~$F$ such that $\varphi(u)\neq\varphi(v)$
(see~\cite{Ribes:1977} for a proof), in which case
$\Gamma$ is said to be profinite.

For another delicate feature of infinite-vertex profinite
semigroupoids,
let~$\Gamma$ be a subgraph of a
topological semigroupoid~$S$,
and let $\lceil \Gamma\rceil$
be the \emph{closed subsemigroupoid of $S$
generated by~$\Gamma$}, that is, $\lceil \Gamma\rceil$
is the intersection of all closed subsemigroupoids of $S$
that contain $\Gamma$. If $S$ has a finite number of vertices, then
$\lceil \Gamma\rceil$ is the topological closure
$\overline{\langle \Gamma\rangle}$
of the discrete
subsemigroupoid $\langle \Gamma\rangle$ of $S$ generated by $\Gamma$.
But if $S$ has an infinite number of vertices, then
$\overline{\langle \Gamma\rangle}$ may not be a semigroupoid and thus it is
strictly contained in $\lceil \Gamma\rceil$~\cite{Almeida&ACosta:2007a}.
If~$\Gamma$ is a profinite graph, then
the
\emph{free profinite semigroupoid generated by~$\Gamma$},
denoted $\Om {\Gamma}{Sd}$,
is a profinite semigroupoid, in which $\Gamma$ embeds as a closed
subgraph, characterized by the property that
every continuous graph homomorphism 
$\varphi\colon \Gamma\to F$ into a finite semigroupoid~$F$ extends
in a unique way to a continuous semigroupoid
homomorphism
$\hat\varphi\colon \Om {\Gamma}{Sd}\to F$.
It turns out that $\lceil \Gamma\rceil=\Om {\Gamma}{Sd}$.
The construction of
$\Om {\Gamma}{Sd}$ is given in~\cite{Almeida&ACosta:2007a} (where some
problems with the construction given in~\cite{Almeida&Weil:1996} are discussed), and
consists in a reduction to the case where $\Gamma$ is finite,
previously treated in~\cite{Jones:1996}.

The free semigroupoid generated by $\Gamma$,
denoted $\Gamma^+$,
is the graph whose vertices are those of $\Gamma$,
and whose edges are the paths of
$\Gamma$ with the obvious composition and incidence laws. The
semigroupoid $\Gamma^+$ embeds naturally in $\Om {\Gamma}{Sd}$, with its elements being topologically
isolated in $\Om {\Gamma}{Sd}$. Moreover, if $\Gamma$ is an inverse
limit $\varprojlim \Gamma_i$ of finite graphs, then
$\Gamma^+=\varprojlim \Gamma_i^+$~\cite{Almeida&ACosta:2007a}.
Also, one has a natural embedding of
$\Om {\Gamma}{Sd}$ in $\varprojlim \Om {\Gamma_i}{Sd}$~\cite{Almeida&ACosta:2007a}.
A problem that we believe remains open and is studied in~\cite{Almeida&ACosta:2007a}, is whether there exists
some example where $\Om {\Gamma}{Sd}\neq \varprojlim \Om
{\Gamma_i}{Sd}$.

Everything we said about semigroupoids has an analog for categories.
We shall occasionally
invoke the free category $\Gamma^\ast$, obtained from $\Gamma^+$ by
adding an empty path $1_v$ at each vertex $v$.

\subsection{Profinite completions of finite-vertex semigroupoids}

A congruence on a semigroupoid $S$ is an equivalence relation
$\theta$ on the set of edges of~$S$
such that $u\mathrel{\theta}v$ implies that
$u$ and $v$ are \emph{coterminal} (that is, they
have the same source and the same target), and
also that $xu\mathrel{\theta} xv$
and $uy\mathrel{\theta} vy$ whenever the products $xu,xv,uy,vy$ are defined.
The quotient $S/{\theta}$ is the semigroupoid with the same set of
vertices of $S$ and edges the classes $u/{\theta}$ with the natural
incidence and composition laws.  The relation that
identifies coterminal edges is a congruence.
Therefore,
if~$S$ has a finite number of vertices, the set $\Lambda$ of congruences
on $S$ such that $S/{\theta}$ is finite is nonempty.
Note that if the congruences $\theta$ and $\rho$
are such that $\theta\subseteq \rho$, then one has a natural semigroupoid
homomorphism $S/{\theta}\to S/{\rho}$. Hence, when $S$ has a
finite number of vertices, we may consider the inverse limit $\hat
S=\varprojlim_{\theta\in\Lambda} S/{\theta}$, which is a profinite
semigroupoid, called the \emph{profinite completion} of $S$.
Let $\iota$ be the natural mapping $S\to\hat S$.
Then $\iota(S)$ is a dense subsemigroupoid of $\hat S$ and $\hat S$ has the property that for every continuous semigroupoid
homomorphism $\varphi$ from $S$ into a profinite semigroupoid $T$
there is a unique continuous semigroupoid homomorphism
$\hat\varphi\colon \hat S\to T$
such that $\hat\varphi\circ\iota=\varphi$~\cite{Jones:1996}.
If $\Gamma$ is a finite-vertex graph, then $\Om {\Gamma}{Sd}$ is the
profinite completion of the free semigroupoid $\Gamma^+$~\cite{Jones:1996}.

\subsection{Profinite groupoids}

A groupoid is a (small) category in which every morphism
has an inverse.
The parallelism between the definitions of profinite semigroups and
profinite groups carries on to an obvious parallelism between
the definitions of topological/profinite semigroupoids and
topological/profinite groupoids.
As groupoids are special cases
of semigroupoids some care is
sometimes needed  when relating
corresponding concepts. The next lemma addresses one of such
situations. For its proof, recall the well known fact that
if $t$ is an element of a compact
semigroup $T$, then
the closed subsemigroup $\overline {\langle t\rangle}$
has a unique
idempotent,
denoted $t^\omega$; in case $T$ is profinite, one has
$t^\omega=\lim t^{n!}$~\cite{Almeida:2002a}.
The inverse of $t\cdot t^\omega$ in
 the maximal subgroup of~$\overline {\langle t\rangle}$
 is denoted $t^{\omega-1}$.

  \begin{Lemma}\label{l:profinite-groupoid-vs-semigroupoid-generated}
    Let $G$ be a compact groupoid and
    suppose that $A$ is a strongly connected subgraph that generates
    $G$ as a topological groupoid. Then $A$ also generates $G$ as a
    topological semigroupoid.
  \end{Lemma}

  \begin{proof}
    Denote by $V_A$ and $V_G$ the vertex sets of $A$ and $G$, respectively.
    Let $H$ be the subgraph of $G$
    with vertex set $\overline{V_A}$
    and whose edges are the edges of $G$ with source and target
    in $\overline{V_A}$.
    Clearly, $H$ is closed and a subgroupoid. Since
    $H$ contains $A$ and $A$ generates $G$ as a topological groupoid,
    we conclude that $H=G$
    and thus $\overline{V_A}=V_G$.

    Consider an arbitrary
    closed subsemigroupoid $S$ of $G$ containing $A$.
    Let~$s$ be an edge of $S$.
    Since $\overline{V_A}=V_G$, there are nets $(a_i)_{i\in I}$
    and $(b_j)_{j\in J}$ of elements of $V_A$
    respectively converging to $\alpha(s)$ and $\omega(s)$.
    Because $A$ is strongly
    connected, for each $(i,j)\in I\times J$
    there is some path $u_{i,j}$ in~$A$ from
    $(b_j)_{j\in J}$ to $(a_i)_{i\in I}$.
    Take an accumulation point $u$ of the net
    $(u_{i,j})_{(i,j)\in I\times J}$.
    Then $u$ is an element of $S$ such that $\alpha(u)=\omega(s)$
    and $\omega(u)=\alpha(s)$.
    In particular,
    we may consider the element $(su)^{\omega-1}$ of the local
    semigroup of $S$ at $\alpha(s)$.
    We claim that $u(su)^{\omega-1}=s^{-1}$.
    Indeed, $s\cdot u(su)^{\omega-1}=(su)^\omega$ is
    the local identity of~$G$ at~$\alpha(s)$, while
    $u(su)^{\omega-1}\cdot s=(us)^\omega$
    is the local identity at~$\omega(s)$.
    Hence $s^{-1}\in S$.
    Since $S$ is an arbitrary closed subsemigroupoid of $G$ containing $A$,
    we conclude that~$s^{-1}$ belongs to the closed
    subsemigroupoid~$K$ of~$G$
    generated by~$A$.
    Therefore, $K$ is a closed subgroupoid of~$G$ containing~$A$. Since $G$ is generated by $A$ as a topological groupoid,
    it follows that $K=G$.
  \end{proof}

  A \emph{groupoid congruence}
  is a semigroupoid congruence $\theta$ on a groupoid
  such that $u\mathrel{\theta}v$ implies $u^{-1}\mathrel{\theta}v^{-1}$.
  If $S$ is a compact groupoid,
  then all closed semigroupoid congruences on $S$ are groupoid congruences.
  Indeed, if $u,v\in S$  are coterminal edges
  then $v^{-1}=u^{-1}(vu^{-1})^{\omega-1}$,
  and if moreover $u\mathrel{\theta}v$, then
  $u^{-1}(vu^{-1})^k\mathrel{\theta} u^{-1}$ for every integer $k\geq 1$,
  whence $v^{-1}\mathrel{\theta}u^{-1}$.

  Replacing semigroupoid congruences by groupoid congruences,
  one gets the notion
  of \emph{profinite completion of a finite-vertex groupoid}
  analogous to the corresponding one for semigroupoids.
  These notions generalize the more familiar
  ones of profinite completion of a group and of a semigroup, since
  (semi)groups are the one-vertex (semi)groupoids.
  The following lemma relates these concepts.

  \begin{Lemma}\label{l:completion-vs-fundamental group}
  Let $G$ be a connected groupoid with finitely many vertices.
  Then the profinite completion of a local group of~$G$ is a
  local group of the profinite completion of~$G$.
\end{Lemma}

\begin{proof}
  Denote by $\hat G$ the profinite completion of~$G$ and let $x$
  be a vertex of~$G$. We must show that the local group $\hat G(x)$ is
  the profinite completion $\widehat{G(x)}$ of the local group~$G(x)$.

  Consider the natural homomorphism $\lambda\colon G\to\hat G$. Note that it
  maps $G(x)$ into the profinite group $\hat G(x)$, which is
  generated, as a topological group, by $\lambda(G(x))$. Thus, the
  restriction $\kappa=\lambda|_{G(x)}$ induces a unique continuous
  homomorphism $\psi:\widehat{G(x)}\to\hat G(x)$, which is onto.

  Suppose that $g\in \widehat{G(x)}\setminus \{1\}$.
  Since $\widehat{G(x)}$
  is a profinite group, there exists a continuous homomorphism
  $\theta\colon \widehat{G(x)}\to H$ onto a finite group $H$ such that
  $\theta(g)\ne1$.
  For each vertex $y$ in~$G$,
  let $p_y\colon x\to y$ be an edge from~$G$. It is easy
  to check that the following relation is a congruence on~$G$: given
  two edges $u,v:y\to z$ in~$G$, $u\sim v$ if
  $\theta\circ\iota(p_yup_z^{-1})=\theta\circ\iota(p_yvp_z^{-1})$. Moreover, note that, in
  case $u,v\in G(x)$, $u\sim v$ if and only if
  $\theta\circ\iota(u)=\theta\circ\iota(v)$.
  Therefore, if $S=G/{\sim}$, then $S(x)$ is finite, whence,
  since $S$ is a connected groupoid, $S$ is finite.
  As $\hat G$ is the profinite completion of~$G$,
  it follows that the natural quotient mapping $\gamma\colon G\to S$
  factors through $\lambda$ as a continuous homomorphism
  $\gamma'\colon\hat G\to S$.
  The restriction $\hat G(x)\to S(x)$ of $\gamma'$ is
  denoted by~$\gamma''$.

  Noting that $\theta\circ\iota$ is onto because
  the image of $\iota$ is dense, and since
  \begin{equation*}
    \theta\circ\iota(u)=\theta\circ\iota(v)
    \iff
    u\sim v
    \iff
    \gamma(u)=\gamma(v),
  \end{equation*}
  there is an isomorphism $\varphi\colon H\to S(x)$ such that
  $\varphi\circ\theta\circ\iota=\gamma|_{G(x)}=\gamma''\circ\psi\circ\iota$.
  Again because
  the image of $\iota$ is dense,
  we deduce that
  $\varphi\circ\theta=\gamma''\circ\psi$.
  
  All these morphisms are represented in
  Diagram~\eqref{eq:diagram-local-local}.
      \begin{equation}
          \begin{split}
    \label{eq:diagram-local-local}
      \xymatrix{ %
    &&\\
    G %
    \ar[r]_\lambda %
    \ar@/^3em/[rr]_\gamma %
    & %
    \hat G %
    \ar[r]_{\gamma'} %
    & %
    S \\
    G(x) %
    \ar@{^(->}[u] %
    \ar[r]^\kappa %
    \ar[rd]_\iota %
    & %
    \hat G(x) %
    \ar@{^(->}[u] %
    \ar[r]^{\gamma''} %
    & %
    S(x) %
    \ar@{^(->}[u] %
    \\
    & %
    \widehat{G(x)} %
    \ar[u]_\psi %
    \ar[r]^\theta %
    &H\ar@{.>}[u]_\varphi
  }
  \end{split}  
\end{equation}
As $\theta(g)\neq 1$, we get $\gamma''\circ\psi(g)=\varphi\circ\theta(g)\neq 1$,whence $\psi(g)\neq 1$. Therefore,
  $\psi$ is an isomorphism of topological groups.
\end{proof}  
  
\subsection{The fundamental groupoid}
\label{sec:fundamental-groupoid}

For the reader's convenience, we write down a definition of
the fundamental groupoid of a graph.
Let $\Gamma$ be a graph.
Extend $\Gamma$ to a graph $\widetilde \Gamma$ by
injectively
associating to each edge $u$
a new formal inverse edge $u^{-1}$ with $\alpha(u^{-1})=\omega(u)$
and $\omega(u^{-1})=\alpha(u)$. One makes $(u^{-1})^{-1}=u$.
Graphs of the form $\widetilde \Gamma$ endowed with the mapping $u\mapsto u^{-1}$
on the edge set are precisely the graphs in the sense of J.-P.~Serre.
These are the graphs upon which the definition of fundamental groupoid
of a graph is built in~\cite{Lyndon&Schupp:1977}, the supporting
reference we give for the next lines.
Consider in the free category $\widetilde\Gamma^\ast$ the
congruence~$\sim$
generated by the identification of $uu^{-1}$ with $1_{\alpha(u)}$ and $u^{-1}u$
with $1_{\omega(u)}$, where $u$ runs over the set of edges of $\Gamma$.
The quotient $\Pi(\Gamma)=\widetilde\Gamma^\ast/{\sim}$ is a groupoid,
called the \emph{fundamental groupoid} of $\Gamma$.
Note that if $\varphi\colon\Gamma_1\to\Gamma_2$
is a homomorphism of graphs, then
the correspondence $\Pi(\varphi)\colon\Pi(\Gamma_1)\to\Pi(\Gamma_2)$
such that $\Pi(\varphi)(x/{\sim})=\varphi(x)/{\sim}$
is a well defined homomorphism of groupoids, and the correspondence
$\varphi\mapsto\Pi(\varphi)$ defines a functor from the category of graphs to the category of groupoids.

It is well known that the natural graph homomorphism
from $\Gamma^\ast$ to $\Pi(\Gamma)$ (that is, the restriction
to $\Gamma^\ast$ of the quotient mapping
$\widetilde\Gamma^\ast\to \Pi(\Gamma)$) is injective.
If $\Gamma$ is connected (as an undirected graph), then
the local groups of $\Pi(\Gamma)$ are isomorphic; their isomorphism
class is the \emph{fundamental group} of $\Gamma$.
It is also well known that if $\Gamma$ is a connected (finite) graph,
then its fundamental group is a (finitely generated) free group.

  \begin{Lemma}
    \label{l:projection-onto-fundamental-groupoid-for-strongly-connected-semigroupoids}
    Let $\Gamma$ be a strongly connected finite-vertex profinite graph.
    Then the natural continuous homomorphism
    from the free profinite semigroupoid $\Om {\Gamma}{Sd}$
    to the profinite completion of $\Pi(\Gamma)$,
    extending the natural graph homomorphism from $\Gamma$
    to $\Pi(\Gamma)$, is onto.\qed
  \end{Lemma}

  To prove Lemma~\ref{l:projection-onto-fundamental-groupoid-for-strongly-connected-semigroupoids} one uses the following fact
  \cite[Corollary 3.20]{Almeida&ACosta:2007a}.

  \begin{Lemma}\label{l:image-of-semigroupoid-homomorphism}
    Let $\psi\colon S\to T$ be a countinuous homomorphism of compact
    semigroupoids. Let $X$ be a subgraph of $S$.
    Then $\psi(\lceil X \rceil)\subseteq \lceil\psi(X)\rceil$.
    Moreover, $\psi(\lceil X \rceil)= \lceil\psi(X)\rceil$
    if $\psi$ is injective on the set of vertices of $S$.
  \end{Lemma}

  \begin{proof}[Proof of
    Lemma~\ref{l:projection-onto-fundamental-groupoid-for-strongly-connected-semigroupoids}]
    Denote by
    $\hat\Pi(\Gamma)$ the profinite completion of $\Pi(\Gamma)$
    and by $h$ the natural
    continuous semigroupoid homomorphism
    $\Om {\Gamma}{Sd}\to \hat\Pi(\Gamma)$.
    By Lemma~\ref{l:image-of-semigroupoid-homomorphism}, the image of $h$
    is the closed subsemigroupoid of $\hat\Pi(\Gamma)$ generated by
    $h(\Gamma)$.
    Since $\hat\Pi(\Gamma)$ is
    generated by $h(\Gamma)$ as a profinite groupoid,
    it follows from Lemma~\ref{l:profinite-groupoid-vs-semigroupoid-generated},
    that $h$ is onto.
  \end{proof}

 \section{Subshifts and their connection with free profinite semigroups}

A subset $X$ of a semigroup $S$ is
\emph{factorial} if the factors of elements
of $X$ also belong to $X$.
The subset $X$ is \emph{prolongable} if
for every $s\in S$ there are $x,y\in X$
such that $xs,sy\in S$.
It is \emph{irreducible}
if for every $u,v\in X$ there is $w\in S$ such that $uwv\in X$.
Using standard compactness arguments, one can show
(see~\cite{ACosta&Steinberg:2011} for a proof) that
if $S$ is a compact semigroup and $X$ is a
nonempty, closed, factorial and irreducible subset
of $S$, then $X$ contains a regular
$\Cl J$-class, called the \emph{apex of $X$} and denoted $J(X)$, such that every element of $X$ is a factor of every element of~$J(X)$.

Let $A$ be a finite set. Endow $A^{\mathbb Z}$ with the product
topology, where $A$ is viewed as discrete space.
Let $\sigma$  be the homeomorphism $A^{\mathbb Z}\to A^{\mathbb Z}$
defined by $\sigma((x_i)_{i\in\mathbb Z})=(x_{i+1})_{i\in\mathbb Z}$,
the \emph{shift} mapping on $A^{\mathbb Z}$.
A \emph{subshift} of $\z A$ is
a nonempty closed subset $\Cl X$ of $\z A$ such that
$\sigma(\Cl X)=\Cl X$. A \emph{finite block} of an element
$x=(x_i)_{i\in\mathbb Z}$ of $\z A$ is a word
of the form $x_ix_{i+1}\ldots x_{i+n}$ (which is also denoted
by $x_{[i,i+n]}$)
for some $n\geq 0$. For a subset $\Cl X$ of $\z A$, denote by
$L(\Cl X)$ the set of finite blocks of elements of $\Cl X$. Then the
correspondence $\Cl X\mapsto L(\Cl X)$ is an isomorphism between the
poset of subshifts of $\z A$ and the poset of factorial, prolongable
languages of $A^+$~\cite[Proposition 1.3.4]{Lind&Marcus:1996}.
A subshift $\Cl X$ is \emph{irreducible} if $L(\Cl X)$ is irreducible.
We are interested in studying the topological closure of $L(\Cl X)$
in $\Om AS$, when $\Cl X$ is a subshift of $\z A$.
It was noticed in~\cite{Almeida&ACosta:2007a}
that $\overline{L(\Cl X)}$ is a
factorial and prolongable subset of $\Om AS$,
and that if $\Cl X$ is irreducible then $\ov{L(\Cl X)}$
is an irreducible subset of $\Om AS$.
Therefore, supposing $\Cl X$ is irreducible, we can consider
the apex $J(\Cl X)$ of $\ov{L(\Cl X)}$. Since
$J(\Cl X)$ is regular, it
has maximal subgroups, which are isomorphic as profinite groups;
we denote by $G(\Cl X)$ the corresponding abstract profinite
group.

In this paper we concentrate our attention on
an important class of irreducible subshifts, the
\emph{minimal} subshifts, that is, those that do
not contain proper subshifts.
This class includes the \emph{periodic} subshifts, finite subshifts
$\Cl X$ for which there is a positive integer $n$ (called a \emph{period})
and $x\in\z A$ such that
$\sigma^n(x)=x$ and $\Cl X=\{\sigma^k(x)\mid 0\leq k<n\}$.
It is well known that a subshift $\Cl X$ is minimal if and only if $L(\Cl X)$
is \emph{uniformly recurrent}, that is, if and only if
for every $u\in L(\Cl X)$ there is an integer $n$
such that every word of $L(\Cl X)$ with length at least $n$
has $u$ as a factor (cf.~\cite[Theorem 1.5.9]{Lothaire:2001}).

For a subshift $\Cl X$ of $\z A$, denote by $\Mir(\Cl X)$ the set
of elements $u$ of $\Om AS$ such that all finite factors of $u$ belong
to $L(\Cl X)$.  One has $\overline{L(\Cl X)}\subseteq \Mir(\Cl X)$,
and there are simple examples of irreducible
subshifts where this inclusion is strict~\cite{Costa:2006}.
In what follows,
a \emph{maximal regular element} of $\Om AS$ is a regular
element of $\Om AS$ that is $\Cl J$-equivalent with its regular
factors. The maximal regular elements of $\Om AS$ are precisely the
elements of $\Om AS\setminus A^+$ all of whose proper factors belong to $A^+$.

\begin{Thm}\label{t:JX-in-minimal-case}
  Let $\Cl X$ be a minimal subshift.
  Then $\overline{L(\Cl X)}=\Mir(\Cl X)$
  and $\overline{L(\Cl X)}\setminus A^+=J(\Cl X)$.
  The correspondence $\Cl X\mapsto J(\Cl X)$ is a bijection between the
  set of minimal subshifts of $\z A$ and the set of
  $\Cl J$-classes of maximal regular elements of $\Om AS$.
\end{Thm}

Theorem~\ref{t:JX-in-minimal-case} is from~\cite{Almeida:2005c}.
In~\cite{Almeida&ACosta:2007a},
an approach whose tools are recalled in the next section,
distinct from that of~\cite{Almeida:2005c}, was used to deduce the equalities 
$\overline{L(\Cl X)}=\Mir(\Cl X)=J(\Cl X)\cup L(\Cl X)$,
when $\Cl X$ is minimal.

A fact that we shall use quite often is that every element of
$\Om AS\setminus A^+$
has a unique prefix in $A^+$
with length~$k$,
and a unique suffix in $A^+$ with length~$k$,
for every $k\geq 1$ (cf.~\cite[Section 5.2]{Almeida:1994a}).
Let $\ZZ^+_0$ and $\ZZ^-$ be respectively the sets of nonnegative integers and of negative integers.
For $u\in\Om AS\setminus A^+$,
we denote by $\ori u$ the
unique element $(x_i)_{i\in\ZZ^+_0}$ of $A^{\mathbb Z^+_0}$ such
that $x_{[0,k]}$ is a prefix of $u$, for every $k\geq 0$,
and by
$\ole u$ the
unique element $(x_i)_{i\in\ZZ^-}$ of $A^{\mathbb Z^-}$ such
that $x_{[-k,-1]}$ is a suffix of $u$, for every $k\geq 1$.
Finally, we denote by $\li u$ the element of
$\z A$ that restricts in $A^{\mathbb Z^-}$
to $\ole u$ and in $A^{\mathbb Z^+_0}$
to $\ori u$.

The part of the next lemma about
Green's relations $\mathcal R$ and $\mathcal L$ was observed in~\cite{Almeida:2003a}
and in~\cite[Lemma 6.6]{Almeida&ACosta:2007a}.
The second part, about the $\mathcal H$ relation, is an easy
consequence of the first part, and it
is proved in a more general context in~\cite[Lemma 5.3]{Almeida&ACosta:2012}.

\begin{Lemma}\label{l:parametrization-of-R-L-classes}
  Let $\Cl X$ be a minimal subshift.  
  Two elements $u$ and $v$ of $J(\Cl X)$ are
  $\Cl R$-equivalent (respectively, $\Cl L$-equivalent) if and only if
  $\ori u=\ori v$ (respectively, $\ole u=\ole v$).
  Moreover, if $x\in\Cl X$,
  then the $\Cl H$-class $G_x$
formed by the elements $u$ of $J(\Cl X)$
such that $\li u=x$ is a maximal subgroup of $J(\Cl X)$.
\end{Lemma}

We retain for the rest of the paper the notation $G_x$ given in
Lemma~\ref{l:parametrization-of-R-L-classes}.

\section{Free profinite semigroupoids generated by
  Rauzy graphs}

Let $\Cl X$ be a subshift of $\z A$.
The \emph{graph of $\Cl X$} is the graph $\Sigma(\Cl X)$
having~$\Cl X$ as the set of vertices and where the edges are precisely
the pairs $(x,\sigma(x))$, with source and target being
respectively equal to $x$ and $\sigma(x)$.
The graph $\Sigma(\Cl X)$ is a compact graph,
with the topology on the edge set being naturally induced by
that of $\Cl X$.

Denote by $L_n(\Cl X)$ the set of elements of $A^+$
with length $n$.
The \emph{Rauzy graph of order $n$} of $\Cl X$,
denoted $\Sigma_n(\Cl X)$, is the graph defined by the following data:
the set of vertices is $L_n(\Cl X)$, the set of edges is
$L_{n+1}(\Cl X)$, and
incidence of edges in vertices is given by
$$a_1a_2\cdots a_n %
    \xrightarrow{a_1a_2\cdots a_na_{n+1}}a_2\cdots a_na_{n+1},$$
    where $a_i\in A$.

\begin{Rmk}\label{r:rauzy-strongly-connected}
  If $\Cl X$ is irreducible,
  then $\Sigma_{n}(\Cl X)$ is strongly connected.
\end{Rmk}    

In the case of a Rauzy graph of even order $2n$, we consider
a function $\mu_n$, called \emph{central labeling}, assigning to
each edge $a_1a_2\cdots a_{2n}a_{2n+1}$ ($a_i\in A$) its middle letter $a_{n+1}$.

\begin{Rmk}\label{r:words-recognized-by-sigma-n}
  Extending the labeling $\mu_n$
  as a semigroupoid homomorphism $\Sigma_{2n}(\Cl X)^+\to A^+$,
one sees that the set of images of
paths of $\Sigma_{2n}(\Cl X)$ by that homomorphism
is the set of elements of $A^+$ whose factors
of length at most $2n+1$ belong to $L(\Cl X)$.
\end{Rmk}

For $m\geq n$, we define a graph homomorphism
    $p_{m,n}\colon \Sigma_{2m}(\Cl X)\to\Sigma_{2n}(\Cl X)$
    as follows: if $w\in L_{2m}(\Cl X)\cup L_{2m+1}(\Cl X)$
    and if $w=vuv'$
    with $v,v'\in A^{m-n}$,
    then $p_{m,n}(w)=u$.
Note that $p_n$ preserves the central labeling, that is,
$\mu_n\circ p_{m,n}(w)=\mu_{m}(w)$ for every edge $w$
of $\Sigma_{2m}(\Cl X)$.
The family of onto graph homomorphisms $\{p_{m,n}\,|\,n\leq m\}$
defines an inverse system of compact graphs.
The corresponding inverse limit $\varprojlim \Sigma_{2n}(\Cl X)$
will be identified with $\Sigma(\Cl X)$
since the mapping from $\Sigma(\Cl X)$ to $\varprojlim\Sigma_{2n}(\Cl X)$
sending $x\in\Cl X$ to $(x_{[-n,n-1]})_n$
and $(x,\sigma(x))$ to $(x_{[-n,n]})_n$
is a continuous graph isomorphism.
The projection $\Sigma(\Cl X)\to \Sigma_{2n}(\Cl X)$
is denoted by $p_n$. Let $\mu$ be the mapping defined on the set of edges of
$\Sigma(\Cl X)$ by assigning $x_0$ to $(x,\sigma(x))$.
Then $\mu=\mu_n\circ p_n$, for every $n\geq 1$.

We proceed with the setting of~\cite{Almeida&ACosta:2007a}.
Like in that paper, denote by $\hat\Sigma_{2n}(\Cl X)$
and by $\hat \Sigma(\Cl X)$
the free profinite semigroupoids generated respectively by
$\Sigma_{2n}(\Cl X)$ and by $\Sigma(\Cl X)$.
The graph homomorphism
$p_{m,n}\colon\Sigma_{2m}(\Cl X)\to\Sigma_{2n}(\Cl X)$ extends
uniquely to a continuous homomorphism
$\hat p_{m,n}\colon\hat\Sigma_{2m}(\Cl X)
\to \hat\Sigma_{2n}(\Cl X)$
of compact semigroupoids.
This establishes an inverse limit
$\varprojlim \hat\Sigma_{2n}(\Cl X)$ in the category of
compact semigroupoids, in which
the graph $\Sigma(\Cl X)=\varprojlim \Sigma_{2n}(\Cl X)$ naturally
embeds.
The canonical projection
$\varprojlim \hat\Sigma_{2n}(\Cl X)\to \hat\Sigma_{2k}(\Cl X)$
is denoted~$\hat p_k$.
Recall that the free profinite semigroupoid $\hat\Sigma(\Cl X)$ also embeds in
$\varprojlim \hat\Sigma_{2n}(\Cl X)$,
and that we do not know of any example where
the inclusion is strict.

\begin{Thm}[{\cite{Almeida&ACosta:2007a}}]\label{t:free-profinite-semigroupoids-in-minimal-case}
  If $\Cl X$ is a minimal subshift then
  $\hat\Sigma(\Cl X)=\varprojlim \hat\Sigma_{2n}(\Cl X)
  =\overline{\Sigma(\Cl X)^+}$.
\end{Thm}

In~\cite{Almeida&ACosta:2007a} one finds examples
of irreducible subshifts $\Cl X$
for which one has $\overline{\Sigma(\Cl X)^+}\neq \hat\Sigma(\Cl X)$.

Viewing $A$ as a virtual one-vertex graph, whose edges are the elements
of~$A$,
the graph homomorphism
$\mu_n\colon\Sigma_{2n}(\Cl X)\to A$
    extends in a unique way to a continuous semigroupoid homomorphism
    $\hat\mu_n\colon\hat\Sigma_{2n}(\Cl X)\to\Om AS$.
    The equality $\mu_m=\mu_n\circ p_{m,n}$ yields
    $\hat\mu_n\circ \hat p_{m,n}=\hat\mu_m$, when $m\geq n\geq 1$,
    and so we may consider the continuous
    semigroupoid
    homomorphism
    $\hat\mu\colon\varprojlim\hat\Sigma_{2n}(\Cl X)\to\Om AS$
    such that $\hat\mu=\hat\mu_n\circ\hat p_n$ for every~$n\geq 1$.
    Recall that a graph homomorphism
    is \emph{faithful} if distinct coterminal edges have distinct images.
    It turns out that $\hat\mu_n$ is
    faithful (cf.~\cite[Proposition 4.6]{Almeida&ACosta:2007a})
    and therefore so is $\hat\mu$.

    Let us now turn our attention to the images of $\hat\mu_n$
    and $\hat\mu$.
    For a positive integer $n$, let
$\Mir_n(\Cl X)$ be the set of all elements $u$ of $\Om AS$
such that all factors of $u$ with length at most $n$ belong
to $L(\Cl X)$.
 
  \begin{Lemma}\label{l:labels-of-hatsigma2n}
    Let $\Cl X$ be a subshift.
    For every positive integer $n$,
    the equality $\hat\mu_n(\hat\Sigma_{2n}(\Cl X))=\Mir_{2n+1}(\Cl X)$ holds.
  \end{Lemma}

  \begin{proof}
                We clearly have
                $\hat\mu_n(\Sigma_{2n}(\Cl X)^+)=\Mir_{2n+1}(\Cl X)\cap A^+$
                (cf.~Remark~\ref{r:words-recognized-by-sigma-n}).
                Noting that $\Mir_{2n+1}(\Cl X)$ is closed and open,
                that $A^+$ is  dense in $\Om AS$,
          and that $\Sigma_{2n}(\Cl X)^+$ is dense in
      $\hat\Sigma_{2n}(\Cl X)$, the lemma  follows immediately.
    \end{proof}

Note that
  $\Mir_1(\Cl X)
  \supseteq
  \Mir_2(\Cl X)
  \supseteq
  \Mir_3(\Cl X)
  \supseteq
  \cdots$  
  and $\Mir(\Cl X)=\bigcap_{n\geq 1}\Mir_n(\Cl X)$.
  Therefore, the image of $\hat\mu$ is contained in
  $\Mir(\Cl X)$, by Lemma~\ref{l:labels-of-hatsigma2n}.
  One actually has $\hat\mu(\varprojlim \hat\Sigma_{2n}(\Cl X))=\Mir(\Cl X)$
  (cf.~\cite[Proposition 4.5]{Almeida&ACosta:2007a}),
  but we shall not need this fact.

  The next two lemmas where observed
    in \cite[Lemmas 4.2 and 4.3]{Almeida&ACosta:2007a}.
    We introduce some notation. We denote by
    $|u|$ the length of a word in $A^+$, and let $|u|=+\infty$
    for $u\in\Om AS\setminus A^+$.

    \begin{Lemma}\label{l:facto-obvio}
      Consider a subshift $\Cl X$.  
      Let $q\colon x_{[-n,n-1]}\to y_{[-n,n-1]}$ be an edge of
      $\hat\Sigma_{2n}(\Cl X)$, where $x,y\in\Cl X$.
  Let $u=\hat\mu_n(q)$.
  If $k=\min\{|u|,n\}$ then $x_{[0,k-1]}$ is a prefix
  of $u$ and $y_{[-k,-1]}$ is a suffix of $u$.
\end{Lemma}

\begin{Lemma}\label{l:vertices-eti}
  Consider a subshift $\Cl X$.
  Let $q\colon x\to y$ be an edge of $\varprojlim \hat\Sigma_{2n}(\Cl X)$.
  Let $u=\hat\mu(q)$. If
  $u\in\Om AS\setminus A^+$ then
  $\ori u=(x_i)_{i\in\ZZ^+_0}$ and
  $\ole u=(y_i)_{i\in\ZZ^-}$.
  If $u\in A^+$ then
  $q$ is the unique edge of $\si X^+$ from $x$ to $\sigma^{|u|}(x)$.
\end{Lemma}

    We denote by $\Pi_{2n}(\Cl X)$
    the fundamental groupoid of $\Sigma_{2n}(\Cl X)$,
    and by $h_n$ the natural homomorphism
$\Sigma_{2n}(\Cl X)\to \Pi_{2n}(\Cl X)$.
The graph homomorphism
$p_{m,n}\colon \Sigma_{2m}(\Cl X)\to \Sigma_{2n}(\Cl X)$
induces the groupoid
homomorphism $q_{m,n}=\Pi(p_{m,n})\colon \Pi_{2m}(\Cl X)\to \Pi_{2n}(\Cl X)$,
characterized by the equality $q_{m,n}\circ h_m=h_n\circ p_{m,n}$.
Let $\hat\Pi_{2n}(\Cl X)$ be the profinite completion of
$\Pi_{2n}(\Cl X)$,
and  let $\hat h_n\colon \hat\Sigma_{2n}(\Cl X)\to \hat\Pi_{2n}(\Cl X)$
and $\hat q_{m,n}\colon \hat\Pi_{2m}(\Cl X)\to \hat\Pi_{2n}(\Cl X)$ be the
natural homomorphisms respectively induced by $h_n$ and $q_{m,n}$.
Then the following diagram commutes:
  \begin{equation}\label{eq:definition-of-h}
    \begin{split}
          \xymatrix@C=4em@R=3em{
   \hat\Sigma_{2m}(\Cl X)
   \ar[d]_{\hat h_m}
   \ar[r]^{\hat p_{m,n}}
   &
   \hat\Sigma_{2n}(\Cl X)   
   \ar[d]^{\hat h_n}
   \\
   \hat\Pi_{2m}(\Cl X)
   \ar[r]_{\hat q_{m,n}}
   &\hat\Pi_{2n}(\Cl X).
    }
    \end{split}
  \end{equation}
  The family $(\hat q_{m,n})_{m,n}$ defines an inverse system of profinite
  groupoids.
We denote by $\hat h$ the
continuous semigroupoid homomorphism from $\varprojlim\hat\Sigma_{2n}(\Cl X)$
to $\varprojlim\hat\Pi_{2n}(\Cl X)$ established by the
commutativity of Diagram~\eqref{eq:definition-of-h}.

For the remainder of this paper, we need to deal with the local
semigroups of the various semigroupoids defined in this section.
Given $n$, we denote respectively by
$\Sigma_{2n}(\Cl X,x)^+$, $\hat\Sigma_{2n}(\Cl X,x)$,
$\Pi_{2n}(\Cl X,x)$, $\hat\Pi_{2n}(\Cl X,x)$ the local semigroups at vertex
$p_{2n}(x)=x_{[-n,n-1]}$ of $\Sigma_{2n}(\Cl X)^+$,
$\hat\Sigma_{2n}(\Cl X)$, $\Pi_{2n}(\Cl X)$ and $\hat\Pi_{2n}(\Cl X)$.

\begin{Rmk}\label{r:profinite-completion-of-local-group}
  If $\Cl X$ is irreducible, then
  $\hat\Pi_{2n}(\Cl X,x)$ is the profinite completion of
  the fundamental group of the strongly connected
  graph $\Sigma_{2n}(\Cl X)$
  (cf.~Lemma~\ref{l:completion-vs-fundamental group}).
\end{Rmk}

\section{Return words in the study
  of \protect{$G(\Cl X)$ in the minimal case}}

Consider a subshift $\Cl X$ of $\z A$.
Let $u\in L(\Cl X)$.
The \emph{return words}\footnote{What we call \emph{return words}
  is sometimes in the literature
  designated \emph{first return words},
  as is the case of
  the article~\cite{Berthe&Felice&Dolce&Leroy&Perrin&Reutenauer&Rindone:2015},
  which is cited later in this paper.
  The terminology that we adopt appears
  for instance
  in~\cite{Durand:1998,Durand&Host&Skau:1999,Balkova&Pelantova&Steiner:2008}.}
of $u$ in~$\Cl X$ are the elements of the set $R(u)$
of words $v\in A^+$ such that $vu\in L(\Cl X)\cap uA^+$
and such that $u$ occurs in $vu$ only as both prefix and suffix.
The characterization
of minimal subshifts via the notion of uniform recurrence
yields that the subshift $\Cl X$ is minimal if and only if,
for every $u\in L(\Cl X)$, the set  $R(u)$ is finite.

Let $n\geq 0$ be such that $|u|\geq n$. Consider words $u_1$ and $u_2$
with $u=u_1u_2$ and $|u_1|=n$.
Let $R(u_1,u_2)$ be the set of words $v$
such that $u_1vu_2\in L(\Cl X)$ and $u_1v\in R(u)u_1$.
In other words, we have $R(u_1,u_2)=u_1^{-1}(R(u)u_1)$. In particular,
$R(u_1,u_2)$ and $R(u)$ have the same cardinality.
The elements of $R(u_1,u_2)$ are callled
in~\cite{Almeida&ACosta:2013}
\emph{$n$-delayed return words of~$u$} in~\Cl X, and
\emph{return words of $u_1.u_2$} in \cite{Durand&Host&Skau:1999}.
Note that $R(u_1,u_2)$ is a code (actually, a circular
code~\cite[Lemma 17]{Durand&Host&Skau:1999}).

Fix $x\in \Cl X$. Denote by $R_n(x)$ the set
$R(x_{[-n,-1]},x_{[0,n-1]})$.
Clearly, if $\Cl X$ is a periodic subshift with period $N$,
then the elements of $R_n(x)$ have length at most $N$.
On the other hand, we have the following result.

\begin{Lemma}[{cf.~\cite[Lemma
    3.2]{Durand:1998}}]\label{l:length-of-return-words-tends-to-infinity}
  If $\Cl X$ is a minimal non-periodic subshift
  then $\lim_{n\to\infty}\min \{|r|: r\in R_n(x)\}=\infty$ for every
  $x\in\Cl X$.\footnote{Lemma~\ref{l:length-of-return-words-tends-to-infinity}
is taken from~\cite[Lemma 3.2]{Durand:1998},
but the limit which appears explicitly
in~\cite[Lemma 3.2]{Durand:1998} 
    is  $\lim_{n\to\infty}\min \{|r|: r\in R(x_{[0,n-1]})\}=\infty$.
    However, $R(z,t)$ is clearly contained in
    the subsemigroup of $A^+$ generated by $R(t)$.
  In particular,
    $\min \{|r|: u\in R_n\}\geq \min \{|r|: u\in R(x_{[0,n-1]})\}$,
  and so our formulation of
  Lemma~\ref{l:length-of-return-words-tends-to-infinity}
  follows immediately from the one in~\cite[Lemma 3.2]{Durand:1998}.}
\end{Lemma}

Let $u\in R_n(x)$.
The word $w=x_{[-n,-1]}ux_{[0,n-1]}$
belongs to $L(\Cl X)$. Its prefix
and its suffix of length $2n$
is the word $x_{[-n,n-1]}$.
Hence, the graph $\Sigma_{2n}(\Cl X)$ has a cycle $s$
rooted at the vertex $x_{[-n,n-1]}$
such that $\mu_n(s)=u$. Since $\mu_n$ is faithful, we may
therefore define a function
$\lambda_n\colon R_n(x)\to\Sigma_{2n}(\Cl X,x)^+$ such that
$\mu_n\circ\lambda_n$ is the identity $1_{R_n(x)}$ on $R_n(x)$.

To extract consequences from these facts at the level
of the free profinite semigroup $\Om AS$, we use the following
theorem from \cite{Margolis&Sapir&Weil:1995}.

\begin{Thm}\label{t:codes-generate-free-profinite-groups}
  If $X$ is a finite code of $A^+$, then the closed subsemigroup of $\Om AS$
  generated by $X$ is a profinite semigroup freely generated by $X$.
\end{Thm}

Assuming that $\Cl X$ is a minimal subshift,
as we do throughout along this section, the
code $R_n(x)$ is finite.
Therefore it follows from
Theorem~\ref{t:codes-generate-free-profinite-groups}
that the profinite subsemigroup $\ov{\langle R_n(x)\rangle}$ of $\Om AS$
is freely generated by $R_n(x)$,
and so the mapping $\lambda_n$ extends in a unique way to a
continuous homomorphism
$\hat\lambda_n\colon \ov{\langle R_n(x)\rangle}\to \hat\Sigma_{2n}(\Cl X,x)$
of profinite semigroups.
Note that the following equality holds by definition of
$\lambda_n$:
\begin{equation}\label{eq:mu_n_lambda_n}
  \hat\mu_n\circ\hat\lambda_n=1_{\ov{\langle R_n(x)\rangle}}.
\end{equation}

If $m\geq n$, then the inclusion $R_m(x)\subseteq \langle R_n(x)\rangle$
clearly holds.

\begin{Lemma}\label{l:intersection-limit-of-return-words}
  Let $\Cl X$ be a minimal non-periodic subshift and let
  $x\in\Cl  X$. Then
  we have  $\bigcap_{n\ge 1} \ov{\langle  R_n(x)\rangle}=G_x$.
\end{Lemma}

\begin{proof}  
  Denote by $I$ the intersection 
  $\bigcap_{n\ge 1} \ov{\langle R_n(x)\rangle}$.
    The inclusion $G_x\subseteq I$ appears
    in~\cite[Lemma 5.1]{Almeida&ACosta:2013}.
    Let us show the reverse inclusion.  
  If $w$ is an element of $\langle R_n(x)\rangle$, then
  it labels a closed path of
  $\Sigma_{2n}(\Cl X)$ at $x_{[-n,n-1]}$.
  Therefore, every factor of $w$ of length at most $2n+1$ 
  belongs to $L(\Cl X)$.
  Since $w$ is an arbitrary element of $\langle R_n(x)\rangle$,
  this implies that
  every factor of length at
  most $2n+1$ of an element of $\ov{\langle R_n(x)\rangle}$
  belongs to $L(\Cl X)$. Therefore, if
  $u\in I$,
  then every finite factor of $u$ belongs to $L(\Cl X)$.
  On the other hand,
  by Lemma~\ref{l:length-of-return-words-tends-to-infinity}
  the elements of $I$ do not belong to $A^+$.
  We conclude
  from Theorem~\ref{t:JX-in-minimal-case}
  that $I\subseteq J(\Cl X)$.
  Let $n>0$.
  By Lemma~\ref{l:length-of-return-words-tends-to-infinity},
  there is $m>n$ such that the length of every element of $R_m(x)$
  is greater than $n$.
  Since the elements of $R_m(x)$ label closed paths at
  $x_{[-m,m-1]}$, we know that
  $R_m(x)\subseteq x_{[0,n-1]}A^+\cap A^+ x_{[-n,-1]}$.
  Hence, we have
  $I\subseteq \ov{\langle R_m(x)\rangle}\subseteq x_{[0,n-1]}\Om
  AS\cap \Om AS x_{[-n,-1]}$. Since $n$ is arbitrary, we deduce
  from the definition of $G_x$ that $I\subseteq G_x$.
\end{proof}

If $\Cl X=\{x\}$ is the singleton periodic subshift
given by $x=\cdots aaa.aaa\cdots$, then $R_n(x)=\{a\}$ for all $n$,
and Lemma~\ref{l:intersection-limit-of-return-words} does not hold in
this case. However, denoting by $\ov{\langle R_n(x)\rangle}_\infty$
the profinite semigroup $\ov{\langle R_n(x)\rangle}\setminus A^+$,
we get the following result,
which can be easily seen
to apply to periodic subshifts.

\begin{Lemma}\label{l:reformulation-intersection-limit-of-return-words}
  Let $\Cl X$ be a minimal subshift and let $x\in \Cl X$. Then
  we have
  $\bigcap_{n\ge 1} \ov{\langle  R_n(x)\rangle}_\infty=G_x$.\qed
\end{Lemma}

  We shall consider the
 inverse systems
 with connecting morphisms the inclusions
 $i_{m,n}\colon\ov{\langle R_m(x)\rangle}\to \ov{\langle
   R_n(x)\rangle}$
 and
$i_{m,n}|\colon\ov{\langle R_m(x)\rangle}_\infty\to
\ov{\langle R_n(x)\rangle}_\infty$.
 Note that we can identify $G_x$ with
  $\varprojlim \ov{\langle R_n(x)\rangle}_\infty$ via
  Lemma~\ref{l:reformulation-intersection-limit-of-return-words}
  (each $g\in G_x$ is identified with the sequence $(g)_{n\ge 1}$).
  Also, one has $G_x\subseteq \varprojlim\ov{\langle R_n(x)\rangle}$, with
  equality in
  the non-periodic case, as seen
  in Lemma~\ref{l:intersection-limit-of-return-words}.

Let $m\geq n$, and let $r\in \ov{\langle R_m(x)\rangle}$.
Then, the equalities
  $$\hat\mu_n(\hat p_{m,n}\circ\hat\lambda_m(r))
  =\hat\mu_m(\hat\lambda_m(r))=r=\hat\mu_n(\hat\lambda_n(r))$$
  hold by~\eqref{eq:mu_n_lambda_n}.
  Since $\hat\mu_n$ is faithful, this shows that the
  following diagram commutes:  
  \begin{equation}\label{eq:definition-of-lambda}
    \begin{split}
          \xymatrix@C=4em@R=3em{
   \ov{\langle R_m(x)\rangle}
   \ar[d]_{\hat\lambda_m}
   \ar@{^{(}->}[r]^{i_{m,n}}
   &
   \ov{\langle R_n(x)\rangle}
   \ar[d]^{\hat\lambda_n}
   \\
   \hat\Sigma_{2m}(\Cl X,x)
   \ar[r]_{\hat p_{m,n}}
   &\hat\Sigma_{2n}(\Cl X,x).
    }
    \end{split}
  \end{equation}
  The commutativity of Diagram~\eqref{eq:definition-of-lambda}
  yields the existence of the homomorphism
  $\hat\lambda=\varprojlim \hat\lambda_n$ from
  $\varprojlim\ov{\langle  R_n(x)\rangle}$ to
  $\varprojlim \hat\Sigma_{2n}(\Cl X,x)$.
  Note that $\varprojlim \hat\Sigma_{2n}(\Cl X,x)$
  is the local semigroup $\hat\Sigma(\Cl X,x)$ of
  $\hat\Sigma(\Cl X)$ at vertex $x$
  (cf.~ Theorem~\ref{t:free-profinite-semigroupoids-in-minimal-case}).

   Let $\hat\Sigma_\infty(\Cl X)$ be
   the subgraph of $\hat\Sigma(\Cl X)\setminus \Sigma(\Cl X)^+$
   obtained by deleting the edges in $\Sigma(\Cl X)^+$.

   \begin{Rmk}\label{r:coincide-in-non-periodic-case}
    When $\Cl X$ is a
   minimal non-periodic subshift, the local semigroup
   $\hat\Sigma_\infty(\Cl X,x)$
   of $\hat\Sigma_\infty(\Cl X)$ at $x$ coincides with
   $\hat\Sigma(\Cl X,x)$.  
   \end{Rmk}
   
   It turns out that $\hat\Sigma_\infty(\Cl X,x)$
   is a profinite group whenever $\Cl X$ is minimal.
   Indeed, the following theorem was announced in~\cite{Almeida:2003b} and
   shown in~\cite[Theorem 6.7]{Almeida&ACosta:2007a}.

   \begin{Thm}\label{t:sigma-infty-is-a-connected-groupoid}
     Let $\Cl X$ be a minimal subshift.
     Then $\hat\Sigma_\infty(\Cl X)$ is a profinite connected groupoid.
   \end{Thm}

   It should be noted that the notion of profiniteness for
   semigroupoids is being taken as compactness plus residual
   finiteness in the category of semigroupoids. If the semigroupoid
   turns out to be a groupoid, one may ask whether profiniteness in
   the category of groupoids is an equivalent property. The answer is
   affirmative since it is easy to verify that, if $\varphi:G\to S$ is
   a semigroupoid homomorphism and $G$ is a groupoid, then the
   subsemigroupoid of~$S$ generated by $\varphi(G)$ is a groupoid.

   In the statement of \cite[Theorem 6.7]{Almeida&ACosta:2007a}, it is
   only indicated that $\hat\Sigma_\infty(\Cl X)$ is a connected
   groupoid but we note that, if a compact semigroupoid is a groupoid,
   then edge inversion and the mapping associating to each vertex the
   identity at that vertex are continuous operations. Thus,
   $\hat\Sigma_\infty(\Cl X)$ is in fact a topological groupoid.
   
   A preliminary version of the next theorem was also announced
   in~\cite{Almeida:2003b}, and a proof appears in the doctoral
   thesis~\cite{Costa:2007}. We present here a different proof, based
   on Lemma~\ref{l:intersection-limit-of-return-words}.
  
  \begin{Thm}\label{t:isomorphism}
    For every minimal subshift $\Cl X$ and every $x\in\Cl X$,
    the restriction $\hat\lambda|\colon G_x
    \to\hat\Sigma_\infty(\Cl X,x)$ is an isomorphism.
    Its inverse is the restriction
    $\hat\mu|\colon\hat\Sigma_\infty(\Cl X,x)\to G_x$.
  \end{Thm}

  \begin{proof}
    By Lemma~\ref{l:reformulation-intersection-limit-of-return-words},
    we know that $G_x=\bigcap_{n\ge 1} \ov{\langle  R_n(x)\rangle}_\infty$,
    and so from~\eqref{eq:mu_n_lambda_n}
    we deduce that $\hat\mu\circ\hat\lambda(g)=g$ for every $g\in G_x$.
    This shows in particular that $\hat\lambda(g)$  must be an infinite
    path whenever $g\in G_x$,
    whence $\hat\lambda(G_x)$ is indeed contained in
    $\hat\Sigma_\infty(\Cl X,x)$. It also shows
    that the restriction
     $\hat\mu|\colon\hat\lambda(G_x)\to G_x$ is onto. Such a restriction is also injective,
    as  $\hat\lambda(G_x)\subseteq \hat\Sigma(\Cl X,x)$
    and $\hat\mu$ is faithful.
    Therefore, all it remains to show is
    the equality $\hat\lambda(G_x)=\hat\Sigma_\infty(\Cl X,x)$.
    
    Let $s\in\hat\Sigma_\infty(\Cl X,x)$
    and let $g=\hat\mu(s)$.
    By Theorem~\ref{t:free-profinite-semigroupoids-in-minimal-case},
    $s$ is the limit of a net of finite paths of
    the graph $\Sigma(\Cl X)$. Since the labeling $\hat\mu$
    of finite paths clearly belongs to $L(\Cl X)$, we 
    have $g=\hat\mu(s)\in\overline{L(\Cl X)}\setminus A^+$ by
    continuity of~$\hat\mu$. It follows that $\hat\mu(s)\in J(\Cl X)$
    by~Theorem~\ref{t:JX-in-minimal-case}.
    Since $s$ is a loop rooted at~$x$,
    applying Lemma~\ref{l:vertices-eti},
    we conclude that $\hat\mu(s)\in G_x$.
    Hence, we have $\hat\mu(s)=g=\hat\mu(\hat \lambda (g))$.
  As $\hat\mu$ is faithful, we get
  $s=\hat\lambda (g)$, concluding the proof.
  \end{proof}

The notion of isomorphism between subshifts is called \emph{conjugacy}.
If $\Cl X$ and~$\Cl Y$ are conjugate subshifts, then
$\Sigma(\Cl X)$ and $\Sigma(\Cl Y)$ are isomorphic, which combined
Theorem~\ref{t:isomorphism} leads to the following result.

\begin{Cor}\label{c:invariance}
  If $\Cl X$ and $\Cl Y$ are conjugate minimal subshifts, then
  the profinite groups $G(\Cl X)$ and $G(\Cl Y)$ are isomorphic.\qed
\end{Cor}

Actually, a more general result was proved in~\cite{Costa:2006}
using different techniques:
if $\Cl X$ and $\Cl Y$ are conjugate irreducible subshifts, then
the profinite groups $G(\Cl X)$
  and $G(\Cl Y)$ are isomorphic.

\section{An application: a sufficient condition for freeness}

In this section, we establish the next theorem, where $FG(A)$ denotes the
free group generated by $A$.

\begin{Thm}\label{t:return-words-form-basis}
  Let $\Cl X$ be a minimal non-periodic subshift, and take $x\in \Cl
  X$. Let $A$ be the set of letters occurring in $\Cl X$. Suppose
  there is a subgroup $K$ of $FG(A)$ and an infinite set $P$ of
  positive integers such that, for every $n\in P$, the set $R_n(x)$ is
  a free basis of $K$. Let $\overline{K}$ be the topological closure
  of~$K$ in $\Om AG$. Then the restriction to $G_x$ of the canonical
  projection $p_{\pv G}\colon \Om AS\to\Om AG$ is a continuous
  isomorphism from $G_x$ onto $\overline{K}$.
\end{Thm}

The following proposition, taken
from~\cite[Proposition 5.2]{Almeida&ACosta:2013},
plays a key role in the proof of Theorem~\ref{t:return-words-form-basis}.

\begin{Prop}\label{p:number-of-generators}
  Let $\Cl X$ be a minimal non-periodic subshift of \z A and let
  $x\in\Cl X$. Suppose there are $M\geq 1$
  and strictly increasing sequences
  $(p_n)_n$
  and
  $(q_n)_n$
  of
  positive integers 
  such that $R(x_{[-p_n,-1]},x_{[0,q_n]})$ has exactly $M$ elements
  $r_{n,1},\ldots,r_{n,M}$, for every~$n$. Let $(r_{1},\ldots,r_{M})$
  be an arbitrary accumulation point of the sequence
  $(r_{n,1},\ldots,r_{n,M})_{n}$ in $(\Om AS)^M$. Then
  $\overline{\langle r_{1},\ldots,r_{M}\rangle}$ is the maximal
  subgroup $G_x$ of $J(\Cl X)$.
\end{Prop}

In the proof of Theorem~\ref{t:tree-are-free}
we shall apply the following lemma, whose proof is an easy and elementary
exercise that we omit.

\begin{Lemma}\label{l:intersection-of-onto}
  Let $S_1\supseteq S_2\supseteq S_3\supseteq \cdots $
  be a descending sequence of compact
  subspaces of a compact space $S_1$.
  Suppose that $\varphi\colon S_1\to T$ is a continuous
  mapping  such that $\varphi(S_n)=T$ for every
  $n\geq 1$.
  If $I=\bigcap_{n\geq 1} S_n$, then we have $\varphi(I)=T$.
\end{Lemma}

We shall also use the following tool.

\begin{Prop}[{\cite[Corollary~2.2]{Coulbois&Sapir&Weil:2003}}]
  \label{p:coulbois-et-al}
  Suppose that $B$ is the basis of a finitely generated subgroup
  $K$ of $FG(A)$. Let $\overline{K}$ be the topological closure of $K$ in
  $\Om AG$. Then $\overline{K}$ is a free profinite group with basis $B$.
\end{Prop}

We are ready to prove Theorem~\ref{t:return-words-form-basis}.

\begin{proof}[Proof of Theorem~\ref{t:return-words-form-basis}]
  By Lemma~\ref{l:intersection-limit-of-return-words},
  we have $G_x=\bigcap_{n\in P}\ov{\langle R_n(x)\rangle}$.
  On the other hand, for every $n\in P$, since by hypothesis
  the set $p_{\pv G}(R_n(x))=R_n(x)$ is a basis of $K$,
  we have $p_{\pv G}\bigl(\,\ov{\langle R_n(x)\rangle}\,\bigr)=\overline{K}$.
  It then follows from
  Lemma~\ref{l:intersection-of-onto}
  that $p_{\pv G}(G_x)=\overline{K}$.

  By assumption, for every $n\in P$, the set $R_n(x)$ has $M$ elements,
  where $M$ is the rank of $K$. Therefore, by
  Proposition~\ref{p:number-of-generators},
  we know that $G_x$ is generated by $M$ elements.
  On the other hand, $\overline{K}$ is a free profinite group of rank~$M$,
  by Proposition~\ref{p:coulbois-et-al}.
  Hence, there is a continuous onto homomorphism
  $\psi\colon \overline{K}\to G_x$. We may then consider the
  continuous onto endomorphism $\varphi$ of $\overline{K}$ such that
  $\varphi(g)=p_{\pv G}(\psi(g))$ for every $g\in \overline{K}$.
  Every onto continuous endomorphism of a finitely generated profinite
  group is an isomorphism~\cite[Proposition 2.5.2]{Ribes&Zalesskii:2000},
  whence $\varphi$ is an isomorphism.
  Since $\psi$ is onto, we conclude that $\psi$ is an isomorphism.
  This shows that the restriction $p_G|\colon G_x\to \overline{K}$
  is the continuous isomorphism $\varphi\circ\psi^{-1}\colon G_x\to \overline{K}$.
\end{proof}

We proceed to apply Theorem~\ref{t:return-words-form-basis}
and two of the main results
of~\cite{Berthe&Felice&Dolce&Leroy&Perrin&Reutenauer&Rindone:2015}
to deduce the freeness of the Sch\"utzenberger group
of the minimal subshifts satisfying the \emph{tree condition}, which we next
describe.

Let $\Cl X$ be a subshift of $\z A$.
Given $w\in L(\Cl X)\cup\{1\}\subseteq A^\ast$, let
\begin{align*}
  L_w&=\{a\in A\mid aw\in L(\Cl X)\},\\
  R_w&=\{a\in A\mid wa\in L(\Cl X)\},\\
  E_w&=\{(a,b)\in A\times A\mid awb\in L(\Cl X)\}.\\
\end{align*}
The \emph{extension graph $G_w$} is the bipartite undirected graph whose
vertex set is the union of disjoint copies of $L_w$ and $R_w$,
and whose edges are the pairs $(a,b)\in E_w$, with incidence in
$a\in L_w$ and $b\in R_w$. The subshift $\Cl X$ satisfies the tree
condition if $G_w$ is a tree for every
$w\in L(\Cl X)\cup\{1\}$.

The class of subshifts satisfying the tree condition contains two
classes that have deserved strong attention in the literature: the
class of \emph{Arnoux-Rauzy subshifts\footnote{The Arnoux-Rauzy
    subshifts over two-letter alphabets are the extensively studied
    \emph{Sturmian subshifts}, but we warn that in
    \cite{Berthe&Felice&Dolce&Leroy&Perrin&Reutenauer&Rindone:2015}
    the Arnoux-Rauzy subshifts are called Sturmian.}} (see the
survey~\cite{Glen&Justin:2009}), and the class of subshifts defined by
\emph{regular interval exchange transformations}~(see
\cite{Berthe&Felice&Dolce&Leroy&Perrin&Reutenauer&Rindone:2015,Berthe&Felice&Dolce&Leroy&Perrin&Reutenauer&Rindone:2015b}).

It is shown in~\cite[Theorem
4.5]{Berthe&Felice&Dolce&Leroy&Perrin&Reutenauer&Rindone:2015} that if
the minimal subshift $\Cl X$ satisfies the tree condition, then, for
every $w\in L(\Cl X)$, the set of return words $R(w)$ is a basis of
the free group generated by the set of letters occurring in $\Cl X$.
This result is called the Return Theorem
in~\cite{Berthe&Felice&Dolce&Leroy&Perrin&Reutenauer&Rindone:2015}.
Combining the Return Theorem with
Theorem~\ref{t:return-words-form-basis}, noting that, for every
$x\in\Cl X$, the set $R_n(x)$ is conjugate to $R_n(x_{[-n,n-1]})$, we
immediately deduce the following theorem.

\begin{Thm}\label{t:tree-are-free}
  If $\Cl X$ is a minimal subshift satisfying the tree condition,
  then $G(\Cl X)$ is a free profinite group with rank $M$, where
  $M$ is the number of letters occurring in $\Cl X$.\qed
\end{Thm}

There are other cases of minimal subshifts \Cl X, not satisfying the
tree condition, for which $G(\Cl X)$ is known to be a free profinite group.
Indeed, it is shown in~\cite{Almeida:2005c}
that if $\Cl X$ is the subshift defined by a
weakly primitive substitution $\varphi$ which  is group invertible,
then $G(\Cl X)$ is a free profinite group.
The weakly primitive substitution
\begin{equation*}
\varphi(a)=ab,\quad
\varphi(b)=cda,\quad\varphi(c)=cd,\quad\varphi(d)=abc
\end{equation*}
is group invertible, but the minimal subshift defined by $\Cl X$ is a
subshift that fails the tree condition~\cite[Example
3.4]{Berthe&Felice&Dolce&Leroy&Perrin&Reutenauer&Rindone:2015}.

The special case of Theorem~\ref{t:tree-are-free}
in which the subshift is an Arnoux-Rauzy subshift was
previously established in~\cite{Almeida:2005c} by the first author
by extending the case of substitution Arnoux-Rauzy subshifts, for which the
substitutions are group invertible.

\section{The groupoids $K_n(\Cl X)_E$}

Let $\Cl X$ be a subshift of $\z A$.
      For every positive integer $n$,
  if $\Cl X_n$ is the subshift of $\z A$ consisting on those
    elements $x$ of $\z A$ such that $x_{[k,k+n-1]}\in L(\Cl X)$
    for every $k\in\ZZ$, then one clearly has $L(\Cl X_n)=\Mir_n(\Cl X)\cap A^+$.
    Since $\Mir_n(\Cl X)$ is a clopen subset of \Om AS,
    it follows that $\overline {L(\Cl X_n)}=\Mir_n(\Cl X)$.
    From this fact one deduces the following lemma. For the sake of uniformity,
    we denote $\Mir(\Cl X)=\bigcap_{n\geq 1}\Mir_n(\Cl X)$
    by $\Mir_\infty(\Cl X)$.

    \begin{Lemma}\label{l:irreducibility-of-mirage}
      For every $n\in\ZZ^+\cup\{\infty\}$, the set $\Mir_n(\Cl X)$ is
      irreducible.
    \end{Lemma}

    \begin{proof}
      Clearly, for every $n\ge1$, if $\Cl X$ is irreducible then so is
      $\Cl X_n$, whence $\Mir_n(\Cl X)=\overline{L(\Cl X_n)}$ is
      irreducible. Let $u,v\in\Mir_\infty(\Cl X)$. For each $n\ge1$,
      there is $w_n\in \Mir_n(\Cl X)$ such that $uw_nv\in\Mir_n(\Cl
      X)$. If $w$ is an accumulation point of $(w_n)_{n\in\ZZ^+}$ then
      $w\in\Mir_n(\Cl X)$ for every $n$, since $\Mir_n(\Cl X)$ is
      closed and $w_m\in \Mir_n(\Cl X)$ for every $m\geq n$. This
      shows $\Mir_\infty(\Cl X)$ is irreducible.
    \end{proof}

    In view of Lemma~\ref{l:irreducibility-of-mirage},
    and since clearly $\Mir_n(\Cl X)$ is closed and factorial
    (irrespectively of $\Cl X$ being irreducible or not),
    we may consider the apex $K_n(\Cl X)$ of
    $\Mir_n(\Cl X)$ when $\Cl X$ is irreducible.

    The irreducibility of $\Cl X$ also
    implies that, for every positive integer $n$,
    the semigroupoid
    $\hat\Sigma_{2n}(\Cl X)$ is strongly connected,
    since $\Sigma_{2n}(\Cl X)$ is then itself strongly connected.
  
  A subsemigroupoid $T$ of a semigroupoid $S$
  is an \emph{ideal}  if for every $t\in T$ and every
  $s\in S$, $\omega(s)=\alpha(t)$ implies $st\in T$, 
  and $\omega(t)=\alpha(s)$ implies $ts\in T$.
  In a strongly connected compact semigroupoid $S$, there is an underlying \emph{minimum ideal $\KerS S$}.
  This ideal $\KerS S$ may be defined as follows. Consider
any vertex $v$ of $S$
and the local semigroup $S(v)$ of $S$ at~$v$.
Then $S(v)$ is a compact semigroup, and therefore
it has a minimum ideal $K_v$. Let $\KerS S$ be the
subsemigroupoid of~$S$ with the same
set of vertices as $S$ and whose edges are those edges of~$S$ that admit
some (and therefore every) element of~$K_v$ as a factor.
Note that $K_v=(\KerS S)(v)$.

The next lemma is folklore. The relations
$\leq_{\Cl J}$ and $\mathcal J$ in semigroupoids extend
naturally the corresponding notions for semigroups,
namely, in a semigroupoid $s\leq_{\Cl J}t$ means the edge $t$ is a factor of the edge $s$.

\begin{Lemma}
  \label{l:minimum-ideal-semigroupoid}
  If $S$ is a strongly connected compact semigroupoid, then
  $\KerS S$ is a closed ideal of $S$ that does not depend on the
  choice of~$v$. Moreover,
  the edges in $\KerS S$ are $\Cl J$-equivalent in~$S$; more precisely,
  they are $\leq_{\Cl J}$-below every edge of $S$.
\end{Lemma}

We next relate $\KerS{\hat\Sigma_{2n}(\Cl X)}$
with $K_{2n+1}(\Cl X)$.

  \begin{Lemma}\label{l:labels-in-the-minimal-ideal}
    Consider an irreducible subshift $\Cl X$ and a positive integer $n$.
    Then we have the equality
    $\hat\mu_n(\KerS{\hat\Sigma_{2n}(\Cl X)})=K_{2n+1}(\Cl X)$.
  \end{Lemma}

  \begin{proof}
    Let $s\in \KerS{\hat\Sigma_{2n}(\Cl X)}$
    and let $w\in K_{2n+1}(\Cl X)$.
    
    By Lemma~\ref{l:labels-of-hatsigma2n}, there is
    $t\in \hat\Sigma_{2n}(\Cl X)$
    such that $\hat\mu_n(t)=w$.
    But $t$ is a factor of $s$ by
    Lemma~\ref{l:minimum-ideal-semigroupoid}, and so $w$
    is a factor of $\hat\mu_n(s)$.
    Again by Lemma~\ref{l:labels-of-hatsigma2n},
    we have $\hat\mu_n(s)\in\Mir_{2n+1}(\Cl X)$.
    The $\leq_{\Cl J}$-minimality of $K_{2n+1}(\Cl X)$ then yields
    $\hat\mu_n(s)\in K_{2n+1}(\Cl X)$,
    establishing the inclusion 
    $\hat\mu_n(\KerS{\hat\Sigma_{2n}(\Cl X)})\subseteq K_{2n+1}(\Cl X)$.
    
    On the other hand, since $\hat\Sigma_{2n}(\Cl X)$ is strongly connected,
    there is an edge $r$ in $\hat\Sigma_{2n}(\Cl X)$
    having $s$ has a factor and such that $tr$ is a loop.
    Let $\ell=(tr)^\omega$.
    Since $\KerS{\hat\Sigma_{2n}(\Cl X)}$ is an ideal,
    we have $\ell\in \KerS{\hat\Sigma_{2n}(\Cl X)}$,
  and so the idempotent $\hat\mu_n(\ell)$ belongs to $K_{2n+1}(\Cl X)$
  by the already proved inclusion.
  But $w=\hat\mu_n(t)\in K_{2n+1}(\Cl X)$
  is a prefix of
  the idempotent $\hat\mu_n(\ell)$,
  and so $w\mathrel{\Cl R}\hat\mu_n(\ell)$
  by stability of $\Om AS$.
  Hence, we have $w=\hat\mu_n(\ell)w=\hat\mu_n(\ell t)$.
  Since $\ell t\in\KerS{\hat\Sigma_{2n}(\Cl X)}$, this shows
  the reverse inclusion
  $K_{2n+1}(\Cl X)\subseteq\hat\mu_n(\KerS{\hat\Sigma_{2n}(\Cl X)})$.
\end{proof}

\begin{Cor}\label{c:existence-of-heavy-loop}
  Let $\Cl X$ be an irreducible subshift.
  Fix a positive integer $n$.
  For every vertex $v$ of $\hat\Sigma_{2n}(\Cl X)$,
  there is an idempotent loop $\ell$ of $\hat\Sigma_{2n}(\Cl X)$ rooted
  at $v$ such that $\hat\mu(\ell)\in K_{2n+1}(\Cl X)$.
\end{Cor}

\begin{proof}
  The graph $\hat\Sigma_{2n}(\Cl X)$ is strongly connected, and so
  every element of $\KerS{\hat\Sigma_{2n}(\Cl X)}$ is a factor of a
  loop $q$ rooted at $v$.
  The loop $\ell=q^\omega$ then satisfies the desired conditions,
  by Lemma~\ref{l:labels-in-the-minimal-ideal}.
\end{proof}

Let $S$ be a semigroup. The category $S_E$ is defined by the following
data:
\begin{enumerate}
\item the vertex set is the set of idempotents of $S$;
\item the edges from $e$ to $f$ are the triples $(e,u,f)$
  with $u\in eSf$;
\item the composition is defined by $(e,u,f)(f,v,g)=(e,uv,g)$.
\end{enumerate}
Note that $(e,e,e)$ is a local identity at each idempotent $e$ of~$S$.
The category $S_E$ was introduced in semigroup theory by Tilson in
his fundamental paper~\cite{Tilson:1987}.
Since the construction $S\mapsto S_E$ is functorial,
if $S$ is profinite,
then $S_E$ becomes a profinite category by considering
the product topology in $S\times S\times S$. In this paper we are interested
in dealing with the profinite category $(\Om AS)_E$.
For an irreducible subshift $\Cl X$ and $n\in\ZZ^+\cup\{\infty\}$,
denote by $K_n(\Cl X)_E$ the subgraph of $(\Om AS)_E$
whose vertices are the idempotents of $K_n(\Cl X)$
and whose edges are the edges $(e,u,f)$
of $(\Om AS)_E$ with $u\in K_n(\Cl X)$.

\begin{Prop}\label{p:JX_E-is-a-category}
  Let $\Cl X$ be an irreducible subshift.
  For every $n\in\ZZ^+\cup\{\infty\}$,
  the graph $K_n(\Cl X)_E$ is a closed subcategory of~$(\Om AS)_E$.
  Moreover, $K_n(\Cl X)_E$ is a profinite groupoid.
\end{Prop}

\begin{proof}
  We know that $K_n(\Cl X)_E$ is
  topologically closed
  in $(\Om AS)_E$ because
  the set of idempotents of $\Om AS$
  and every $\Cl J$-class of $\Om AS$
  are closed.

  As shown in~\cite[Lemma 8.2]{Almeida&Volkov:2006}, if $w$ is a finite factor of a product
  $pqr$ with $p,q,r\in\Om AS$ and $q\notin A^+$,
  then $w$ is a factor of $pq$ or of $qr$.
  Therefore, the composition in $(\Om AS)_E$ of two edges
  of $K_n(\Cl X)_E$ belongs to $K_n(\Cl X)_E$,
  and so $K_n(\Cl X)_E$ is a subcategory of $(\Om AS)_E$.

  If $(e,u,f)$ is an edge of $K_n(\Cl X)_E$,
  then $e\mathrel{\Cl R}u\mathrel{\Cl L}f$ by stability of~\Om AS.
  It follows from the basic properties of Green's relations
  that there is some $v$ in $K_n(\Cl X)$
  such that $f\mathrel{\Cl R}v\mathrel{\Cl L}e$,
  $uv=e$ and $vu=f$. Hence $(f,v,e)$
  is an edge of $K_n(\Cl X)_E$
  that is an inverse of $(e,u,f)$, thereby establishing that
  $K_n(\Cl X)_E$ is a groupoid.

  To conclude the proof, it remains to show that $K_n(\Cl X)_E$ is
  residually finite as a topological groupoid. Since it is a
  subgroupoid of the category $(\Om AS)_E$, which is residually finite
  as a topological category, the topological groupoid $K_n(\Cl X)_E$
  is residually finite as the subcategory generated by the image of a
  homomorphism of a topological groupoid into a finite category is
  easily seen to be a groupoid.
\end{proof}

In the minimal case, we may combine
Proposition~\ref{p:JX_E-is-a-category} and Theorem~\ref{t:isomorphism}
to obtain an alternative characterization of the profinite groupoid
$\hat\Sigma_\infty(\Cl X)$ in terms of the local structure of the free
profinite semigroup \Om AS. For this purpose, we introduce some
notation that is also useful in the next section.

Suppose $\Cl X$ is a minimal subshift. For each $x\in\Cl X$, let
$\ell_x$ be the identity at $x$ in the groupoid $\hat\Sigma_\infty(\Cl
X)$ (cf.~Theorem~\ref{t:sigma-infty-is-a-connected-groupoid}). Let
$e_x$ be the idempotent $\hat\mu(\ell_x)$. Recall that $e_x$ is the
identity element of $G_x$ (cf.~Theorem~\ref{t:isomorphism}).

\begin{Rmk}\label{rmk:continuity-of-ex}
  For every minimal subshift,
  the mapping  $x\in\Cl X\mapsto \ell_x\in \hat\Sigma_\infty(\Cl X)$
  is continuous, and therefore so is
  the mapping $x\in\Cl X\mapsto e_x\in J(\Cl X)$.
\end{Rmk}

By Theorem~\ref{t:JX-in-minimal-case}, we know that $K_\infty(\Cl
X)=J(\Cl X)$. By Proposition~\ref{p:JX_E-is-a-category}, we know that
$J(\Cl X)_E$ is a profinite groupoid. Note that for each $x\in\Cl X$,
the profinite groups $G_x$ and the local group of $J(\Cl X)_E$ are
isomorphic, the mapping $u\in G_x\mapsto (e_x,u,e_x)$ being a continuous
isomorphism between them. The following gives a sort of first
geometric characterization of the groupoid $J(\Cl X)_E$.

\begin{Thm}\label{t:globalized-version-of-labeling-isomorphism}
  For every minimal subshift $\Cl X$, we have a continuous groupoid
  isomorphism $F\colon \hat\Sigma_\infty(\Cl X)\to J(\Cl X)_E$ defined
  on vertices by $F(x)=e_x$ and on edges by $F(s)
  =(e_{\alpha(s)},\hat\mu(s),e_{\omega(s)})$.
\end{Thm}

\begin{proof}
  Note first that $F$ is clearly a functor between categories, as $\hat\mu$
  is itself a semigroupoid homomorphism.
  The continuity of $F$ follows from
  the continuity of $\hat\mu$ and Remark~\ref{rmk:continuity-of-ex}.
  Let $e$ be an idempotent of $J(\Cl X)$, and take $x=\li e$.
  Since $e_x\in G_x$, we have $e_x=e$ in view of
  Lemma~\ref{l:parametrization-of-R-L-classes},
  whence $F(x)=e$. On the other hand,
  if $F(x)=F(y)$, then $x=y$, also in view of Lemma~\ref{l:parametrization-of-R-L-classes}. This establishes that $F$ is bijective on vertices.

  Fix an element $x\in\Cl X$.
  Consider the isomorphism 
  $u\in G_x\mapsto (e_x,u,e_x)$,
  from $G_x$ onto the local group of $J(\Cl X)_E$ at $e$.
  Composing it with the restriction of $\hat\mu$ to the local group
  $\hat\Sigma_\infty(\Cl X,x)$ we get,
  thanks to Theorem~\ref{t:isomorphism},
  a continuous isomorphism, which is precisely the restriction
  of~$F$ mapping $\hat\Sigma_\infty(\Cl X,x)$ onto the local group of $J(\Cl X)_E$  at $e$.

  Finally, it is an easy exercise to show that if $H$ is
  a functor between two
  connected groupoids $S$ and $T$
  that restricts to a bijection between the
  corresponding sets of vertices and to a bijection between some local group
  of $S$ and some local group of $T$, then $H$ is an isomorphism of groupoids.
\end{proof}

The following lemma is useful in the sequel.

\begin{Lemma}
  \label{l:convergence-of-apex-idempotents}
  Let $\Cl X$ be an irreducible subshift.
  If $e$ is an idempotent in $K_\infty(\Cl X)$,
  then there is a sequence $(e_n)_n$
  of idempotents $e_n\in K_n(\Cl X)$ such that $\lim e_n=e$.
\end{Lemma}

\begin{proof}
  For each positive integer $n$, choose
  $v_n\in K_n(\Cl X)$.
  Since $e\in\Mir_n(\Cl X)$ and $\Mir_n(\Cl X)$ is irreducible,
  there are $z_n,t_n\in\Om AS$
  such that $ez_nv_nt_ne$ belongs to $\Mir_n(\Cl X)$,
  whence $(e,ez_nv_nt_ne,e)$ is a loop of $K_n(\Cl X)_E$,
  and so is $(e,ez_nv_nt_ne,e)^\omega$
  in view of Proposition~\ref{p:JX_E-is-a-category}.
  Therefore, the idempotent $e_n=(ez_nv_nt_ne)^\omega$ belongs to
  $K_n(\Cl X)$.

  Let $f$ be
  an accumulation point of the sequence $(e_n)_n$.
  Note that $f$ is an idempotent such that
  $f\leq_{\Cl R}e$ and $f\leq_{\Cl L}e$.
  As $m\geq n$ implies $e_m\in \Mir_n(\Cl X)$
  and because
  $\Mir_n(\Cl X)$ is closed, we have $f\in\Mir_n(\Cl X)$ for every~$n\ge1$,
  whence $f\in\Mir_\infty(\Cl X)$.
  Therefore, since $e\in K_\infty(\Cl X)$ is a factor of $f$,
  we must have $f\in K_\infty(\Cl X)$.
  As $\Om AS$ is stable, we conclude that $f=e$.
  We have shown that $e$ is the unique
  accumulation point of $(e_n)_n$,
  and so by compactness we conclude that $(e_n)_n$ converges to $e$.  
\end{proof}

\section{A geometric interpretation of $G(\Cl X)$ when $\Cl X$ is minimal}

In this section we present a series of technical results
that culminate, for the case where $\Cl X$ is a minimal subshift,
in the geometric interpretation of $G(\Cl X)$
as an inverse limit of the profinite completions of the fundamental groups of the
Rauzy graphs
$\Sigma_{2n}(\Cl X)$  (Corollary~\ref{c:the-geometric-interpretation-reformulation}).
While some preliminary results are valid for all irreducible
subshifts, 
we leave open whether our main result generalizes to that case.

By Corollary~\ref{c:existence-of-heavy-loop},
if $\Cl X$ is an irreducible subshift then,
for each vertex $w$ of $\Sigma_{2n}(\Cl X)$,
we may choose an idempotent loop $\ell_{w,n}$
of $\hat\Sigma_{2n}(\Cl X)$ rooted
at $w$ such that the idempotent $e_{w,n}=\hat\mu_n(\ell_{w,n})$
belongs to $K_{2n+1}(\Cl X)$. 

\begin{Lemma}\label{l:convergence-of-evn}
  Suppose $\Cl X$ is a minimal subshift.
    For every $x\in\Cl X$,
    the sequence $(e_{x_{[-n,n-1]},n})_n$
    converges to $e_x$.
 \end{Lemma}  

 \begin{proof}
   Since $\Mir(\Cl X)$ is the intersection
   of the descending chain of closed sets $(\Mir_{2n+1}(\Cl X))_n$,
   we know that every accumulation point $e$ of $(e_{x_{[-n,n-1]},n})_n$
   is an idempotent belonging to $\Mir(\Cl X)$.
   We also know that, for a fixed a positive integer $k$, the word $x_{[0,k]}$ is
   a prefix of $e_{x_{[-n,n-1]},n}$ whenever $n>k$,
   by Lemma~\ref{l:vertices-eti}. By continuity,
   we deduce that $x_{[0,k]}$ is a prefix of $e$.
   Similarly, $x_{[-k,-1]}$ is a suffix of $e$.
   Since $k$ is arbitrary, we conclude from
   Lemma~\ref{l:parametrization-of-R-L-classes} that $e=e_x$.
   Hence, by compactness, the sequence
   $(e_{x_{[-n,n-1]},n})_n$ converges to $e_x$, as $e_x$ is its sole
    accumulation point.
 \end{proof}

 Let $u\in\Om AS$. Suppose $z\in A^+$ is such that $u\in z\cdot \Om AS$. Then there is a unique $w$ in $\Om AS$
 such that $u=zw$~\cite[Exercise 10.2.10]{Almeida:1994a}. We denote $w$ by $z^{-1}u$.
The product $(z^{-1}u)z$ is denoted simply by $z^{-1}uz$, as
there is no risk of ambiguity. Observe that
if $u$ is idempotent then $z^{-1}uz$ is also idempotent.
In terms of the element $x=(x_i)_{i\in\mathbb Z}$ of the minimal
subshift $\mathcal X$, one sees that
$e_x\in x_0\cdot\Om AS$, and so we may consider
the idempotent $x_0^{-1}e_{x}x_0$.

\begin{Lemma}\label{l:conjugacy-of-ex}
  If $\Cl X$ is a minimal subshift, then
  for every $x\in\Cl X$ we have $e_{\sigma(x)}=x_0^{-1}e_{x}x_0$.
 \end{Lemma}

 \begin{proof}
   Let $w=x_0^{-1}e_{x}x_0$. Then we have $\li w=\sigma(x)$. Hence,
   $w$ is an idempotent in $G_{\sigma(x)}$, that is,
   $w=e_{\sigma(x)}$.
 \end{proof}

 By the freeness of the profinite semigroupoid
 $\hat\Sigma(\Cl X)$, we may consider
the unique continuous semigroupoid homomorphism
$\Psi\colon \hat\Sigma(\Cl X)\to \Om AS$ such that
$\Psi(s)=e_{\alpha(s)}\cdot\hat\mu(s)\cdot e_{\omega(s)}$
for every edge $s$ of $\Sigma(\Cl X)$.

\begin{Lemma}\label{l:what-Psi-is}
  Suppose $\Cl X$ is a minimal subshift.
  For every edge $s$ of $\hat\Sigma(\Cl X)$,
  we have
  \begin{equation}\label{eq:what-Psi-is-1}
      \Psi(s)=e_{\alpha(s)}\cdot\hat\mu(s)\cdot e_{\omega(s)}.
  \end{equation}
  Moreover, if $s$ is an infinite edge then
  $\Psi(s)=\hat\mu(s)$.
\end{Lemma}

\begin{proof}
  We first establish equality~\eqref{eq:what-Psi-is-1}
  for finite paths $s$ belonging to $\Sigma(\Cl X)^+$, by
  induction on the length of $s$.
  The base case holds by the definition of $\Psi$.

  Suppose that~\eqref{eq:what-Psi-is-1}
  holds for paths in $\Sigma(\Cl X)$
  of length $k$, where $k\geq 1$, and let~$s$ be a path in
  $\Sigma(\Cl X)$ of length $k+1$.
    Factorize $s$ as $s=tr$ with~$t$ being a path of length
  $1$ and $r$ a path of length $k$. Then, by the induction hypothesis,
  and since $\Psi$ is a semigroupoid homomorphism, we have,
  \begin{equation}\label{eq:image-of-Psi-finite-path}
    \Psi(s)=\Psi(t)\Psi(r)=
    e_{\alpha(s)}\cdot\hat\mu(t)\cdot e_{\omega(t)}
    \cdot\hat\mu(r)\cdot e_{\omega(s)}.
  \end{equation}
  Since $t$ has length $1$, there is
  $x\in\Cl X$ such that $t=(x,\sigma(x))$.
  As $\alpha(s)=x$, $\omega(t)=\sigma(x)$ and $\hat\mu(t)=x_0$,
  and taking into account Lemma~\ref{l:conjugacy-of-ex},
  we obtain
    $e_{\alpha(s)}\cdot\hat\mu(t)\cdot e_{\omega(t)}
    =e_x\cdot x_0 \cdot x_0^{-1}e_xx_0
    =e_{\alpha(s)}\cdot\hat\mu(t)$.
    Hence, \eqref{eq:image-of-Psi-finite-path}
    simplifies to
    \begin{equation*}
        \Psi(s)=
    e_{\alpha(s)}\cdot\hat\mu(t)\cdot\hat\mu(r)\cdot e_{\omega(s)}
    =e_{\alpha(s)}\cdot\hat\mu(s)\cdot e_{\omega(s)},
    \end{equation*}
    which establishes the inductive step, and concludes the proof by induction
    that \eqref{eq:what-Psi-is-1} holds for finite paths.

    Denote by $\Phi$
    the mapping $\hat\Sigma(\Cl X)\to \Om AS$
    such that $\Phi(s)=e_{\alpha(s)}\cdot\hat\mu(s)\cdot e_{\omega(s)}$
    for every edge $s$ of $\hat\Sigma(\Cl X)$.
    We proved that $\Psi$ and $\Phi$ coincide
    in $\Sigma(\Cl X)^+$.
    By continuity of $\hat\mu$
    and by Remark~\ref{rmk:continuity-of-ex},
    we know that $\Phi$ is continuous.
    Hence, as $\Sigma(\Cl X)^+$ is dense
    in $\hat\Sigma(\Cl X)$
    by Theorem~\ref{t:free-profinite-semigroupoids-in-minimal-case},
    we conclude that $\Psi=\Phi$.

    Suppose $s$ is an infinite edge.
    Since $s$ and $\ell_{\alpha(s)}$ have the same source,
    $\hat\mu(s)$ and $e_{\alpha(s)}$ have the same set of finite prefixes
    by Lemma~\ref{l:vertices-eti}. This means that
    $\hat\mu(s)$ and $e_{\alpha(s)}$ are $\mathcal R$-equivalent elements
    of $J(\Cl X)$, by Lemma~\ref{l:parametrization-of-R-L-classes}. Similarly, $\hat\mu(s)$ and $e_{\omega(s)}$
    are $\mathcal L$-equivalent. This establishes $\Psi(s)=\hat\mu(s)$.
\end{proof}

We begin a series of technical lemmas
preparing a result (Proposition~\ref{p:Psi-n-converges-unif-to-Psi})
about the  approximation of $\Psi$
by a special sequence of functions in 
the function space $(\Om AS)^{\hat\Sigma(\Cl X)}$,
endowed with the pointwise topology.

\begin{Lemma}
  \label{l:special-integer-uniformity}
  Suppose $\Cl X$ is a minimal subshift.
  Let $\varphi$ be a continuous semigroup homomorphism
  from $\Om AS$ into a finite semigroup $F$.
  Then there is an integer $N_\varphi$ such that
  if $u$ is an element of $\Mir_{N_\varphi}(\Cl X)$
  with length at least $N_\varphi$,
  then
  $\varphi(u)\in\varphi(J(\Cl X))$.
\end{Lemma}

\begin{proof}
  Since $ J(\Cl X)\subseteq\overline{L(\Cl X)}$,
  there is $z\in L(\Cl X)$ such that $\varphi(z)\in \varphi(J(\Cl X))$.
  By the uniform recurrence of $L(\Cl X)$, there is an integer
  $M$ such that every word of $L(\Cl X)$ of length at least $M$
  contains $z$ as a factor.

  Let $e$ be an idempotent of $J(\Cl X)$.
  Since $\Cl X$ is a minimal subshift,
  by Theorem~\ref{t:JX-in-minimal-case}
  we know that $K_\infty(\Cl X)=J(\Cl X)$.
  Applying Lemma~\ref{l:convergence-of-apex-idempotents}, we
  conclude that there is a sequence $(e_n)_n$ of idempotents converging to $e$
  such that $e_n\in K_n(\Cl X)$ for every $n\ge1$.
  Hence, there is an integer
  $N_\varphi$ with $N_{\varphi}\geq M$
  for which we have $\varphi(e_n)=\varphi(e)$ whenever $n\geq N_\varphi$.

  Let $u\in\Mir_{N_\varphi}(\Cl X)$ be
  such that the length of $u$ is at least $N_\varphi$.
  Then $z$ is a factor of $u$.
  We also have $e_{N_\varphi}\leq_\Cl J u$ by the definition of $K_n(\Cl X)$.
  Hence, we obtain
  $\varphi(e_{N_\varphi})\leq_\Cl J \varphi(u) \leq_\Cl J \varphi(z)$.
  But both $\varphi(z)$ and $\varphi(e_{N_\varphi})=\varphi(e)$
  belong to $\varphi(J(\Cl X))$, thus
  $\varphi(u)\in \varphi(J(\Cl X))$.  
\end{proof}

\begin{Lemma}\label{l:uniform-convergence-of-evn}
  Let $\Cl X$, $\varphi$ and $N_{\varphi}$
  be as in Lemma~\ref{l:special-integer-uniformity}.
   For all $x\in\Cl X$ and $n\geq N_\varphi$,
   the equality $\varphi(e_x)=\varphi(e_{x_{[-n,n-1]},n})$
   holds.
 \end{Lemma}

 \begin{proof}
   By Lemmas~\ref{l:facto-obvio} and~\ref{l:vertices-eti},
   the word  $x_{[0,n-1]}$
   is a common prefix of
   $e_{x_{[-n,n-1]},n}$ and $e_x$. 
   Note also that, for $n\geq N_\varphi$,
   $x_{[0,n-1]}$, $e_{x_{[-n,n-1]},n}$, and $e_x$ belong
   to $\Mir_{N_\varphi}(\Cl X)$.
   In view of Lemma~\ref{l:special-integer-uniformity},
   we conclude that the elements of the set
   \begin{equation*}
   \{\varphi(x_{[0,n-1]}),\varphi(e_{x_{[-n,n-1]},n}),\varphi(e_x)\}     
   \end{equation*}
   belong to $\varphi(J(\Cl X))$.
   By stability of $F$, we deduce that
   \begin{equation*}
     \varphi(e_{x_{[-n,n-1]},n})
     \mathrel{\Cl R}
     \varphi(x_{[0,n-1]})
     \mathrel{\Cl R}
     \varphi(e_x).
   \end{equation*}
   Similarly, we have
   \begin{equation*}
     \varphi(e_{x_{[-n,n-1]},n})
     \mathrel{\Cl L}
     \varphi(x_{[-n,-1]})
     \mathrel{\Cl L}
     \varphi(e_x).
   \end{equation*}
   Hence $\varphi(e_{x_{[-n,n-1]},n})\mathrel{\Cl H}\varphi(e_x)$,
   and since $e_{x_{[-n,n-1]},n}$ and $e_x$ are idempotents,
   we actually have $\varphi(e_{x_{[-n,n-1]},n})=\varphi(e_x)$.   
 \end{proof}

 Let $\Cl X$ be an irreducible subshift.
 Consider the graph homomorphism
$\psi_n\colon \Sigma_{2n}(\Cl X)\to (\Om AS)_E$
defined by
\begin{equation*}
  \psi_n(s)=(e_{\alpha(s),n},e_{\alpha(s),n}\cdot \hat\mu_n(s)\cdot
  e_{\omega(s),n},e_{\omega(s),n})
\end{equation*}
for each edge $s$ of $\Sigma_{2n}(\Cl X)$.
By the freeness of the profinite semigroupoid  $\hat\Sigma_{2n}(\Cl X)$,
the graph homomorphism $\psi_n$ extends in a unique way to a continuous
semigroupoid homomorphism
$\hat\psi_n\colon \hat\Sigma_{2n}(\Cl X)\to (\Om AS)_E$.

\begin{Lemma}\label{l:the-image-is-in-the-groupoid}
  For every irreducible subshift $\Cl X$,
  the image of $\hat\psi_n$ is contained in
  the groupoid $K_{2n+1}(\Cl X)_E$.
\end{Lemma}

\begin{proof}
  Let $s$ be an edge of $\Sigma_{2n}(\Cl X)$.
  By their definition,
  the idempotents $e_{\alpha(s),n}$ and $e_{\omega(s),n}$
  belong to $K_{2n+1}(\Cl X)$.
  Take $u=e_{\alpha(s),n}\cdot \hat\mu_n(s)\cdot e_{\omega(s),n}$.
  We have $u=\hat\mu_n(\ell_{\alpha(s),n}\cdot s\cdot \ell_{\omega(s),n})$.
  Since $\ell_{\alpha(s),n}\cdot s\cdot \ell_{\omega(s),n}$
  belongs to $\hat\Sigma_{2n}(\Cl X)$,
  we must have $u\in\Mir_{2n+1}(\Cl X)$ by
  Lemma~\ref{l:labels-of-hatsigma2n}.
  But $u=e_{\alpha(s),n}\cdot u\cdot e_{\omega(s),n}$, and so $u\in K_{2n+1}(\Cl X)$
  by the $\leq_{\Cl J}$-minimality of $K_{2n+1}(\Cl X)$, establishing that
  $\hat\psi_n(s)$ belongs to $K_{2n+1}(\Cl X)_E$. 
  Since $K_{2n+1}(\Cl X)_E$ is a closed subcategory of
  $(\Om AS)_E$,
  applying     Lemma~\ref{l:image-of-semigroupoid-homomorphism}
  we conclude that the image of $\hat\psi_n$ is contained in $K_{2n+1}(\Cl X)_E$.
\end{proof}

Denote by $\gamma$ the continuous semigroupoid homomorphism
$(\Om AS)_E\to\Om AS$ defined by $\gamma(e,u,f)=u$.
Consider the following sequence of continuous semigroupoid homomorphisms:
\begin{equation*}
  \xymatrix{
    \hat\Sigma(\Cl X)\ar[r]^{\hat p_n}
    &\hat\Sigma_{2n}(\Cl X)\ar[r]^{\hat\psi_n}
    &(\Om AS)_E\ar[r]^{\gamma}
    &\Om AS.
        }
\end{equation*}
Let $\Psi_n=\gamma\circ\hat\psi_n\circ \hat p_n$ be the resulting composite.

For the next proposition, we take into account
the metric $d$ of $\Om AS$
such that if $u$ and $v$ are distinct elements of $\Om AS$,
then $d(u,v)=2^{-r(u,v)}$,
where $r(u,v)$ is the minimum possible cardinality of a finite semigroup~$F$
for which there is a continuous
homomorphism $\varphi\colon \Om AS\to F$
satisfying $\varphi(u)\neq\varphi(v)$. The hypothesis
which we have been using that $A$ is finite guarantees
that the metric $d$ generates the topology of
$\Om AS$~\cite{Almeida:2002a,Almeida:2003b}.

\begin{Prop}\label{p:Psi-n-converges-unif-to-Psi}
  Suppose $\Cl X$ is a minimal subshift.
  Endow the function space $(\Om AS)^{\hat\Sigma(\Cl X)}$
  with the pointwise topology.
  Then the sequence $(\Psi_n)_n$ converges uniformly to $\Psi$.
\end{Prop}

\begin{proof}
  Fix a positive integer $k$.
  We want to show that
  there is a positive integer $N_k$
  such that if $n\geq N_k$
  then $d(\Psi_n(s),\Psi(s))<\frac{1}{2^k}$
  for every $s\in  \hat\Sigma(\Cl X)$.
  For that purpose we use the following auxiliary
  lemma, whose proof is a standard exercise. It appears implicitly in
  the first part of the proof of Proposition 7.4
  from  \cite{Almeida:2002a}.
  
  \begin{Lemma}
    Fix a positive integer $k$.
    There is a continuous
    semigroup homomorphism $\varphi$ from $\Om AS$ onto a finite semigroup $F$
  such that
  \begin{equation}
        \label{eq:ball-homo}
  d(u,v)<\frac{1}{2^k}\Longleftrightarrow \varphi(u)=\varphi(v).
\end{equation}
\end{Lemma}

Proceeding with the proof of Proposition~\ref{p:Psi-n-converges-unif-to-Psi},
let $\varphi:\Om AS\to F$ be a continuous homomorphism onto a finite
semigroup $F$ such that the equivalence~\eqref{eq:ball-homo} holds.
Let $N_\varphi$ be an integer as in
Lemmas~\ref{l:special-integer-uniformity} and~\ref{l:uniform-convergence-of-evn}.
Consider an integer $n$ with $n\geq N_\varphi$.
  In view of equivalence~\eqref{eq:ball-homo},
  the proposition is proved once
  we show that, for every edge $s$ of $\hat\Sigma(\Cl X)$,
  we have
  \begin{equation}
    \label{eq:uniformity-Psi}
    \varphi(\Psi_n(s))=\varphi(\Psi(s)).
  \end{equation}   
If $s$ has length $1$, that is, if $s$ is an edge $(x,\sigma(x))$ of
$\Cl X$, for some $x\in\Cl X$, then we have
\begin{equation}
  \label{eq:p:properties-of-Psi-a}
  \Psi_n(s)=
  e_{\alpha(\hat p_{n}(s)),n}
  \cdot
  \hat\mu_{n}(\hat p_{n}(s))
  \cdot
  e_{\omega(\hat p_{n}(s)),n}.
\end{equation}
Because $n\geq N_\varphi$,
it follows from Lemma~\ref{l:uniform-convergence-of-evn} that
\begin{equation}
  \label{eq:p:properties-of-Psi-b}
\varphi(e_{\alpha(\hat p_{n}(s)),n})=
\varphi(e_{\alpha(s)})\quad
\text{and}\quad
\varphi(e_{\omega(\hat p_{n}(s)),n})=
\varphi(e_{\omega(s)}).  
\end{equation}
Since $\hat\mu_n\circ\hat p_n=\hat\mu$,
from~\eqref{eq:p:properties-of-Psi-a}
and~\eqref{eq:p:properties-of-Psi-b}
we obtain~\eqref{eq:uniformity-Psi} in the case where
$s$ has length $1$.
Hence, $\varphi\circ\Psi_n$ and $\varphi\circ\Psi$
are continuous semigroupoid homomorphisms coinciding in
$\Sigma(\Cl X)$. Since $\Sigma(\Cl X)^+$ is dense
in $\hat\Sigma(\Cl X)$
by Theorem~\ref{t:free-profinite-semigroupoids-in-minimal-case},
it follows that we actually have
$\varphi\circ\Psi_n=\varphi\circ\Psi$, thereby
 establishing~\eqref{eq:uniformity-Psi}.
\end{proof}

Suppose $\Cl X$ is an irreducible subshift.
As the graph $\Sigma_{2n}(\Cl X)$ is strongly connected,
for each edge $s\colon v_1\to v_2$ of $\Sigma_{2n}(\Cl X)$
one can choose a path $s'\colon v_2\to v_1$ in $\Sigma_{2n}(\Cl X)$.
Denote by $s^\ast$ the edge $(s's)^{\omega-1}s'$
of $\hat\Sigma_{2n}(\Cl X)$ from $v_2$ to $v_1$.

\begin{Rmk}\label{rmk:groupoid-image-of-paths}
  For every edge $s$ of $\Sigma_{2n}(\Cl X)$, the loops
  $s^\ast\cdot s$ and $s\cdot s^\ast$ are idempotents.
  Therefore, if $\varphi$ is a semigroupoid homomorphism
  from $\hat\Sigma_{2n}(\Cl X)$ into a groupoid, then
  $\varphi(s^\ast)=\varphi(s)^{-1}$ for every edge $s$.
\end{Rmk}

Recall how in Section~\ref{sec:fundamental-groupoid}
we defined the graph $\widetilde\Gamma$ from a graph $\Gamma$,
and denote $\widetilde{\Sigma_{2n}(\Cl X)}$
by~$\widetilde\Sigma_{2n}(\Cl X)$.
Let $t^{\varepsilon}$
be an edge of  $\widetilde\Sigma_{2n}(\Cl X)$,
where $t$ is an edge of $\Sigma_{2n}(\Cl X)$, $\varepsilon\in\{-1,1\}$
and $t^1=t$.
We define
\begin{equation*}
  (t^\varepsilon)^+=
  \begin{cases}
    t&\text{if $\varepsilon=1$},\\
    t^\ast&\text{if $\varepsilon=-1$}.\\
  \end{cases}
\end{equation*}
If $s=s_1s_2\cdots s_k$ is a path,
where each $s_i$ is an edge of $\widetilde\Sigma_{2n}(\Cl X)$,
then
we define
$s^+=s_1^+ s_2^+\cdots s_k^+$.
Note that $s^+$
is an edge of $\hat\Sigma_{2n}(\Cl X)$
such that $\alpha(s^+)=\alpha(s)$
and $\omega(s^+)=\omega(s)$. We also follow the usual definition
$s^{-1}=s_k^{-1}\cdots s_2^{-1}s_1^{-1}$.

If $1_v$ is the empty path at some vertex $v$
of $\widetilde\Sigma_{2n}(\Cl X)$, then one takes
$1_v=1_v^{-1}=1_v^\ast=1_v^+$,
and if $\varphi$ is a semigroupoid homomorphism
from $\hat\Sigma_{2n}(\Cl X)$ into a groupoid,
then one defines $\varphi(1_v)$ as being the local unit at
$\varphi(v)$.

\begin{Lemma}\label{l:an-exercise-on-hat}
  Consider an irreducible subshift $\Cl X$.
  Let $\varphi$ be a semigroupoid homomorphism
  from $\hat\Sigma_{2n}(\Cl X)$ into a groupoid,
  and let $t$ be a (possible empty) path of $\widetilde\Sigma_{2n}(\Cl X)$.
  Then we have
  \begin{equation}\label{eq:an-exercise-on-hat-1}
      \varphi(t^+)^{-1}
  =\varphi((t^{-1})^+).
  \end{equation}
\end{Lemma}

\begin{proof}
    The case where $t$ is an empty path is immediate.
  We show~\eqref{eq:an-exercise-on-hat-1} by induction
  on the length of $t$.
  Suppose that $t$ has length $1$.
  Either $t\in\Sigma_{2n}(\Cl X)$
  or $t^{-1}\in\Sigma_{2n}(\Cl X)$.
  In the first case we have
  $t^+=t$
  and $(t^{-1})^+=t^\ast$,
  while in the second
  case we have
  $t^+=(t^{-1})^\ast$
  and $(t^{-1})^+=t^{-1}$.
  In either case, \eqref{eq:an-exercise-on-hat-1}
  follows from Remark~\ref{rmk:groupoid-image-of-paths}.

  Suppose that \eqref{eq:an-exercise-on-hat-1}
  holds for paths of length less than $k$, where $k>1$.
  Let  $t$ be a path of $\widetilde\Sigma_{2n}(\Cl X)$ of length
  $k$, and consider a factorization in $\widetilde\Sigma_{2n}(\Cl X)^+$
  of the form $t=rs$ with $s$ an edge of $\widetilde\Sigma_{2n}(\Cl X)$.

  Then $t^+=r^+ s^+$.
  Therefore, applying the inductive hypothesis, we get
  \begin{align*}
    \varphi(t^+)^{-1}&
    =\varphi(s^+)^{-1}\cdot
    \varphi(r^+)^{-1}
    \\
    &=
    \varphi((s^{-1})^+)\cdot
    \varphi((r^{-1})^+)
    =\varphi\bigl((s^{-1}r^{-1})^+)
    =\varphi\bigl((t^{-1})^+),
  \end{align*}
  which establishes the inductive step and concludes the proof.
\end{proof}

Recall that a \emph{spanning tree} $T$ of a graph $\Gamma$
is a subgraph of $\Gamma$ which, with respect to inclusion,
is maximal for the property
that the undirected graph underlying $T$ is both connected and
without cycles. In what follows, we say that a path $t_1\ldots t_k$
of $\widetilde \Gamma$ lies in $T$ if for each $i\in \{1,\ldots,k\}$
we have $t_i\in T$ or $t_i^{-1}\in T$.

Fix an element $x\in\Cl X$
and, for each $n$, fix a spanning tree $T$ of the graph $\Sigma_{2n}(\Cl X)$.
For each pair of vertices $v_1,v_2$ of $\Sigma_{2n}(\Cl X)$,
let $T_{v_1,v_2}$ be the unique path of $\widetilde\Sigma_{2n}(\Cl X)$
from $v_1$ to $v_2$ that does not repeat vertices and lies in~$T$.
Note that $T_{v_1,v_2}=T_{v_2,v_1}^{-1}$, and that $T_{v,v}$ is the empty path at
$v$.
For each vertex $v$
let $\gamma_v= T_{v,x_{[-n,n-1]}}$
and $\delta_v= T_{x_{[-n,n-1]},v}$. In particular, we have
  $\gamma_v=\delta_v^{-1}$.  

  For each edge $s$ of $\Sigma_{2n}(\Cl X)$, consider the element
  $g_s$ of the local group $\Pi_{2n}(\Cl X,x)$ given by $g_s=
  (\delta_{\alpha(s)}\cdot s\cdot \gamma_{\omega(s)})/{\sim}$. Note
  that $g_s$ is the identity of $\Pi_{2n}(\Cl X,x)$ if $s$ belongs
  to~$T$. Denote by $Y$ the set of edges of $\Sigma_{2n}(\Cl X)$ not
  in~$T$. It is a well known fact that the set $B=\{g_s\mid s\in Y\}$
  is a free basis of the fundamental group $\Pi_{2n}(\Cl
  X,x)$~\cite{Lyndon&Schupp:1977}. Hence, $B$ is a basis of the
  free profinite group $\hat \Pi_{2n}(\Cl X,x)$. In view
  of~Lemma~\ref{l:the-image-is-in-the-groupoid}, we may therefore
  consider the unique continuous group homomorphism $\zeta_n$ from
  $\hat \Pi_{2n}(\Cl X,x)$ into the local group of the profinite
  groupoid $K_{2n+1}(\Cl X)_E$ at $e_n=e_{x_{[-n,n-1]},n}$ such that
  \begin{equation*}
    \zeta_n(g_s)=
    \hat\psi_n
    \Bigl(\delta_{\alpha(s)}^+
    \cdot s\cdot
    \gamma_{\omega(s)}^+
    \Bigr)
  \end{equation*}
  for every $s\in Y$.

  \begin{Lemma}
    \label{l:factorization-of-psin}
  Consider an irreducible subshift $\Cl X$.    
    Let $u$ be a loop
    of  $\hat\Sigma_{2n}(\Cl X)$
    rooted at vertex $x_{[-n,n-1]}$.
    Then we have $\zeta_n(\hat h_n(u))=\hat\psi_n(u)$.
  \end{Lemma}
  
  \begin{proof}
    Since we are dealing with finite-vertex graphs,
    we have $\ov{\Sigma_{2n}(\Cl X)^+}=\hat\Sigma_{2n}(\Cl X)$.
    And since the vertex space of  $\hat\Sigma_{2n}(\Cl X)$
    is discrete,  it follows that any loop
    of  $\hat\Sigma_{2n}(\Cl X)$ rooted at $x_{[-n,n-1]}$
    is the limit of a net of finite loops of
    $\hat\Sigma_{2n}(\Cl X)$ rooted at $x_{[-n,n-1]}$.
    Hence, since $\zeta_n\circ\hat h_n$
    and $\hat\psi_n$ are continuous, the lemma is proved
    once we show  that the equality
    $\zeta_n(\hat h_n(u))=\hat\psi_n(u)$  holds whenever $u$
    is a finite loop rooted at $x_{[-n,n-1]}$.
    For such a finite loop $u$,
    let
  \begin{equation*}
  u=u_0s_1u_1s_2u_2\cdots u_{k-1}s_ku_k    
  \end{equation*}
  be a factorization in $\Sigma_{2n}(\Cl X)^+$ such that $u_0,\ldots,
  u_k$ are (possibly empty) paths that lie in~$T$ and $s_1,\ldots,s_k$
  are edges belonging to $Y$. Let $w_i$ be the longest common prefix
  of $\gamma_{\omega(s_i)}^{-1}$ and $\delta_{\alpha(s_{i+1})}$ and let
  $z_i$ and $t_i$ be such that the equalities
  $\gamma_{\omega(s_i)}=z_iw_i^{-1}$ and
  $\delta_{\alpha(s_{i+1})}=w_it_i$ hold in~$\tilde{\Sigma}_{2n}(\Cl X)$.
  Note that
\begin{equation}\label{eq:delta-gamma-u}
  u_0=\delta_{\alpha(s_1)},\;
  u_k=\gamma_{\omega(s_k)},\;\;
  \text{and}\;\;
  u_i=z_it_i
  \quad\text{for $i\in \{1,\ldots,k-1\}$}.
\end{equation}
  It follows that
  \begin{equation*}
    \hat h_n(u)=g_{s_1}g_{s_2}\cdots g_{s_{k}}
  \end{equation*}
  and so
  \begin{equation}\label{eq:zeta_n-h_n}
    \zeta_n(\hat h_n(u))=
\hat\psi_n
\Bigl(
\delta_{\alpha(s_1)}^+\cdot s_1\cdot \gamma_{\omega(s_1)}^+
\cdot
\delta_{\alpha(s_2)}^+\cdot s_2\cdot \gamma_{\omega(s_2)}^+
\cdots
\delta_{\alpha(s_k)}^+\cdot s_k\cdot \gamma_{\omega(s_k)}^+
\Bigr).
\end{equation}
On the other hand, by \eqref{eq:delta-gamma-u},
we have $\delta_{\alpha(s_1)}^+=u_0$, $\gamma_{\omega(s_k)}^+=u_k$
and, in view of Lemmas~\ref{l:the-image-is-in-the-groupoid}
and~\ref{l:an-exercise-on-hat}, for $i\in \{1,\ldots,k-1\}$, the
following chain of equalities holds:
\begin{align*}
  \hat\psi_n(\gamma_{\omega(s_{i})}^+
  \cdot\delta_{\alpha(s_{i+1})}^+)
  &=\hat\psi_n\bigl(z_i^+(w_i^{-1})^+\cdot w_i^+t_i^+\bigr) \\
  &=\hat\psi_n(z_i^+)\cdot
  \hat\psi_n\bigl((w_i^{-1})^+\bigr)\cdot\hat\psi_n(w_i^+)\cdot
  \hat\psi_n(t_i^+) \\
  &=\hat\psi_n(z_i^+)\cdot
  \hat\psi_n(w_i^+)^{-1}\cdot\hat\psi_n(w_i^+)\cdot
  \hat\psi_n(t_i^+) \\
  &=\hat\psi_n(z_i^+)\cdot
  \hat\psi_n(t_i^+)
  = \hat\psi_n(u_{i}).
\end{align*}
Therefore, \eqref{eq:zeta_n-h_n} simplifies to
$\zeta_n(\hat h_n(u))=\hat\psi_n(u)$, as we wished to show.
  \end{proof}

  \begin{Thm}\label{t:the-geometric-interpretation}
    Let $\Cl X$ be a minimal subshift. Then, the restriction of the
    mapping $\hat{h}$ to $\hat{\Sigma}_\infty(\Cl X)$ is an
    isomorphism of topological groupoids onto
    $\varprojlim\hat\Pi_{2n}(\Cl X)$.
  \end{Thm}

  \begin{proof}
    By
    Lemma~\ref{l:projection-onto-fundamental-groupoid-for-strongly-connected-semigroupoids}, $\hat h_n$~is onto
    for every $n\geq 1$,
    which shows that $\hat h$ is
    onto (cf.~\cite[Theorem 29.13]{Willard:1970}).
    Therefore, by Theorem~\ref{t:free-profinite-semigroupoids-in-minimal-case},
    the equality
    $\hat h(\hat\Sigma(\Cl X))=\varprojlim\hat\Pi_{2n}(\Cl X)$ holds.
     If $s$ is a finite edge in
     $\hat\Sigma(\Cl X)$, then $\ell_{\alpha(s)}s$
     is an edge in $\hat\Sigma_\infty(\Cl X)$
     such that $\hat h(\ell_{\alpha(s)}s)=\hat h(s)$,
     whence
     $\hat h(\hat\Sigma_\infty(\Cl X))=\varprojlim\hat\Pi_{2n}(\Cl X)$.
    
    Let $s,t$ be elements of $\hat\Sigma_\infty(\Cl X)$
    such that $\hat h(s)=\hat h(t)$. Since $\hat{h}$ is the identity
    mapping on vertices, we may assume that $s$ and $t$ are edges and,
    therefore, they are coterminal edges.
    Then, for every $n\geq 1$, we have
    $\hat h_n(\hat p_n(ss^{-1}))=\hat h_n(\hat p_n(ts^{-1}))$, and so
    from Lemma~\ref{l:factorization-of-psin}
    we deduce the equality
    \begin{equation*}
      \hat \psi_n(\hat p_n(ss^{-1}))=\hat \psi_n(\hat p_n(ts^{-1})).
    \end{equation*}
    This shows that $\Psi_n(ss^{-1})=\Psi_n(ts^{-1})$
    every $n\geq 1$.
    From
    Proposition~\ref{p:Psi-n-converges-unif-to-Psi}
    we then obtain $\Psi(ss^{-1})=\Psi(ts^{-1})$.
    By Lemma~\ref{l:what-Psi-is}, this means that
    $\hat\mu(ss^{-1})=\hat\mu(ts^{-1})$.
    Since $\hat\mu$ is faithful,
    we conclude that $ss^{-1}=ts^{-1}$, whence $s=t$. This establishes
    that $\hat h$ is injective.
  \end{proof}

  In view of Theorem~\ref{t:isomorphism}, we may now obtain our main
  result as an immediate consequence of
  Theorem~\ref{t:the-geometric-interpretation}.

  \begin{Cor}\label{c:the-geometric-interpretation-reformulation}
    If $\Cl X$ is a minimal subshift
    then $G(\Cl X)$ is isomorphic with $\varprojlim\hat\Pi_{2n}(\Cl X,x)$
    as a profinite group, for every $x\in\Cl X$.\qed
  \end{Cor}

\bibliographystyle{amsplain}

\bibliography{sgpabb,ref-sgps}

\end{document}